\documentclass[preprint]{imsart}

\RequirePackage[OT1]{fontenc}
\RequirePackage{pdfsync}
\RequirePackage{amsfonts}
\RequirePackage[numbers]{natbib}
\RequirePackage[colorlinks,citecolor=blue,urlcolor=blue]{hyperref}
\RequirePackage[centertags, leqno, sumlimits, nointlimits, namelimits]{mathtools}
\RequirePackage[thmmarks,thref,amsthm,amsmath]{ntheorem}
\RequirePackage{paralist}
\RequirePackage[english, noabbrev]{cleveref}

\startlocaldefs
\numberwithin{equation}{section}
\theoremstyle{plain}
\newtheorem{theorem}{Theorem}[section]
\newtheorem{lemma}[theorem]{Lemma}

\theoremstyle{remark}
\newtheorem{remark}[theorem]{Remark}

\endlocaldefs

\newcommand{\N}{\mathbb{N}} 
\newcommand{\R}{\mathbb{R}} 
\newcommand{\lebesgue}{\ensuremath{\lambda\!\!\!\;\!\lambda}} 
\DeclareMathOperator*{\essup}{ess\,sup} 
\begin{document}

\begin{frontmatter}
\title{Signal detection via Phi-divergences for general mixtures}
\runtitle{Signal detection via Phi-divergences for general mixtures}

\begin{aug}
\author{\fnms{Marc} \snm{Ditzhaus}\ead[label=e1]{marc.ditzhaus@uni-duesseldorf.de}}


\runauthor{M. Ditzhaus}

\affiliation{Heinrich-Heine University D\"usseldorf \thanksmark{m1}}

\address{Address of the author\\
Heinrich-Heine University D\"usseldorf\\
Mathematical Institute\\
Universit\"atststra{\ss}e 1\\
40225 D\"usseldorf, Germany\\
\printead{e1}}

\end{aug}

\begin{abstract}
	\textbf{Summary:} In this paper we are interested in testing whether there are any signals hidden in high dimensional noise data. Therefore we study the family of goodness-of-fit tests based on $\Phi$-divergences including  the test of Berk and Jones as well as Tukey's higher criticism test. The optimality of this family is already known for the heterogeneous normal mixture model. We now present a technique to transfer this optimality to more general models. For illustration we apply our results to dense signal and sparse signal models including the exponential-$\chi^2$ mixture model and general exponential families as the normal, exponential and Gumbel distribution. Beside the optimality of the whole family we discuss the power behavior on the detection boundary and show that the whole family has no power there, whereas the likelihood ratio test does.
\end{abstract}


 
\end{frontmatter}

\section{Introduction}\label{sec:intro}
In several research areas it is of interest to detect rare and weak signals hidden in a huge noisy background. These areas are, among others, genomics  \cite{DaiETAL2012,Goldstein2009,IyengarElston2007}, disease surveillance \cite{KulldorffETAL2005,NeillLingwall2007}, local anomaly detection \cite{SaligramaZhao2012} as well as cosmology and astronomy  \cite{CayonETAL2004,JinETAL2005Cosmo}. E.g., in genomics we want to determine as early as possible whether a patient is healthy or affected by a common disease like cancer or leukemia. Many researchers assume that the majority of an affected patients' genes behaves like genes of a non-affected patient (noisy background) and only a small amount of genes displays a slightly different behavior (signals). In other words, if there are any signals at all then they are represented rarely and weakly. This combination makes it very difficult to detect the signals. In this paper we study tests for this detection problem. After introducing the mathematical model we give more details about tests which were already suggested in the literature and explain our new insights into these.   \\
Let $P_n$ be a known continuous noise distribution and $\mu_n$ be an unknown signal distribution on $(\R,{\mathcal B })$. E.g., $\mu_n((-\infty,x])=P_n((-\infty,x-\vartheta_n])$ for $\vartheta_n\in\R$, i.e. a signal leads to a shift by $\vartheta_n$. Now, let $X_{n,1},\ldots,X_{n,n}$ be an i.i.d. sample with     
\begin{align*}
	X_{n,1}\sim Q_n=(1-\varepsilon_{n})P_n+\varepsilon_{n}\mu_n \text{ fo some }\varepsilon_n\in[0,1].	
\end{align*}
The parameter $\varepsilon_{n}$ can be interpreted as the probability that $X_{n,1}$ follows the signal distribution $\mu_n$ instead of the noise distribution $P_n$. Hence, the number of signals is random and approximately the size of $n\varepsilon_{n}$. In this paper we are interested to test whether there are any signals, i.e. to test
\begin{align}\label{eqn:testing_problem}
	{\mathcal H }_{0,n}: \varepsilon_{n}=0\; \text{ versus }\; {\mathcal H }_{1,n}: \varepsilon_{n}>0.
\end{align}
We focus on the challenging case of rare signals $\varepsilon_{n}\to 0$, where we differ between the sparse signal case ($n\varepsilon_n\to 0$) and the dense signal case ($n\varepsilon_n\to \infty$). Throughout this paper, if not stated otherwise, all limits are meant as $n\to\infty$. Clearly, the likelihood ratio test, also called Neyman-Pearson test, is the best test for  \eqref{eqn:testing_problem}. Its power behavior was studied by \citet{Ingster1997} for a normal location model, i.e. $P_n=N(0,1)$ and $\mu_n=N(\vartheta_n,1)$. This model is also called the heterogeneous normal mixture model. Using the parametrization  $\varepsilon_n=n^{-\beta}$, $\beta\in(1/2,1)$, and $\vartheta_n=\sqrt{2r\log(n)}$, $r>0$, he showed that there is a detection boundary $\rho(\beta)$, which splits the $r$-$\beta$-parametrization plane into the completely detectable and the undetectable area, see \Cref{fig:normal}:
\begin{align}\label{eqn:detect_boundary_standnorm}
	\rho(\beta)=\begin{cases}
				\beta-\frac{1}{2} & \textrm{ if }\beta\in(\frac{1}{2},\frac{3}{4}],\\
				(1-\sqrt{1-\beta})^2& \textrm{ if }\beta\in(\frac{3}{4},1).
			\end{cases}
\end{align} 
\begin{figure}[tb] 
	\begin{center}	
		\includegraphics[trim = 22mm 150mm 80mm 10mm, clip, width=0.45\textwidth]{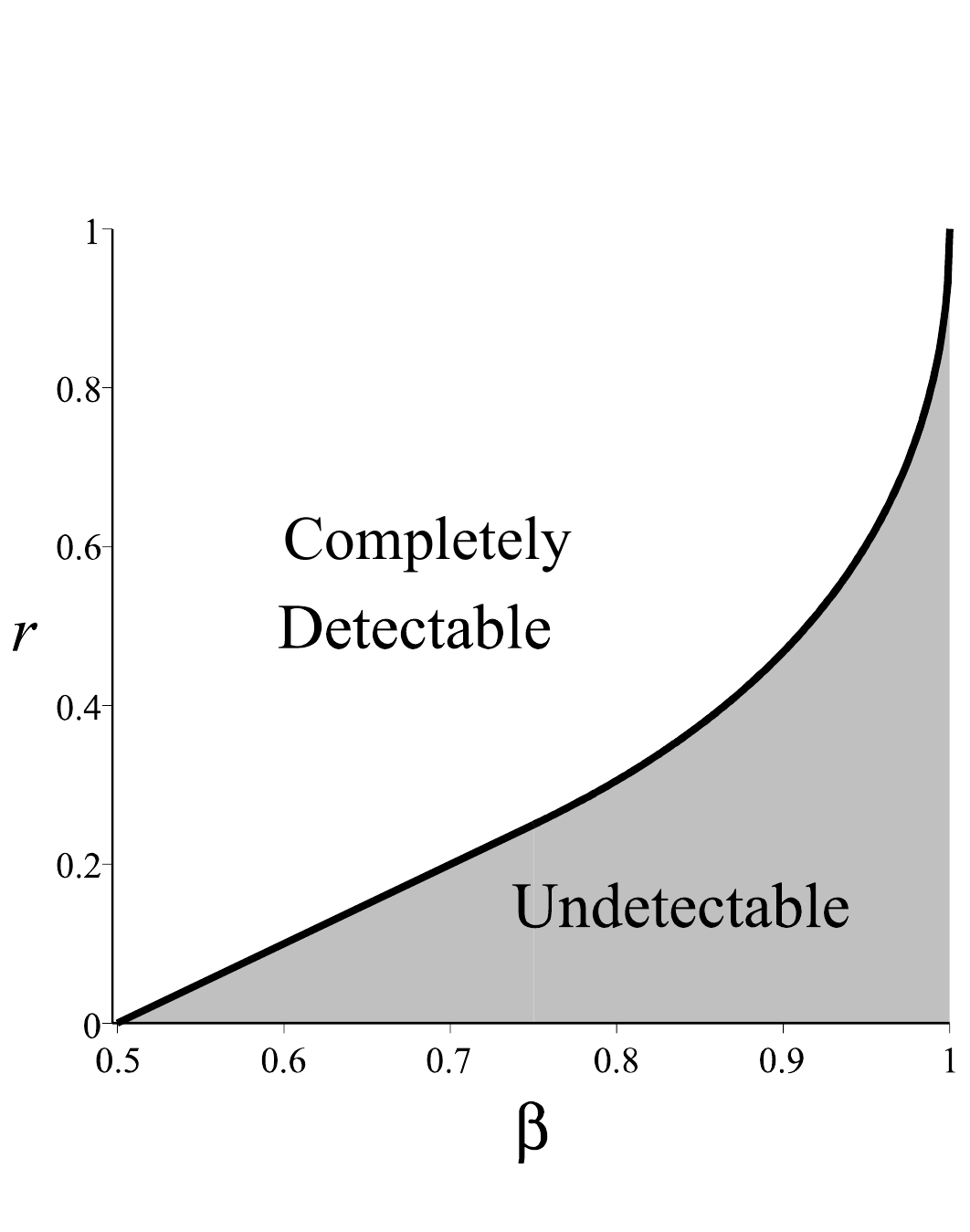}
	\end{center}
	\caption[Detection boundary for the sparse heterogeneous normal mixture model]{
		The detection boundary for the sparse heterogeneous normal mixture model $\beta\mapsto \rho(\beta)$, see \eqref{eqn:detect_boundary_standnorm}, is plotted. It splits the $r$-$\beta$-parametrization plane into the completely detectable and the undetectable area, where the null $\mathcal H_{0,n}$ and the alternative  $\mathcal H_{1,n}$ can be completely separated and merge (asymptotically), respectively. } \label{fig:normal}
\end{figure}\\
If $r>\rho(\beta)$ then the likelihood ratio test can completely separate ${\mathcal H }_{0,n}$ and ${\mathcal H }_{1,n}$ asymptotically (completely detectable case), i.e. the sum of type 1 and 2 error probabilities tends to $0$. Otherwise, if $r<\rho(\beta)$, the likelihood ratio test and, thus, any other test cannot distinguish between ${\mathcal H }_{0,n}$ and ${\mathcal H }_{1,n}$ asymptotically (undetectable case). Later, \citet{DonohoJin2004} showed that Tukey's higher criticism \cite{TukeyCoursenotes,TukeyInternalPaper,TukeyCollected} can also completely separate ${\mathcal H }_{0,n}$ and ${\mathcal H }_{1,n}$ asymptotically if $r>\rho(\beta)$. This showed a certain optimality of the higher criticism test. In contrast to the likelihood ratio test the higher criticism test does not need the knowledge of the typically unknown signal probability $\varepsilon_n$ and signal strength $\vartheta_n$. \citet{JagerWellner2007} introduced a family of test statistics $\{S_n(s):s\in[-1,2]\}$ based on $\Phi$-divergences including  higher criticism test statistic and the test statistic of \citet{BerkJones1979}. They extended the optimality result of \citet{DonohoJin2004} to their whole family. But in contrast to the higher criticism test, see \cite{AriasCandesPlan2015,AriasWang2015,AriasWang2017,CaiJengJin2011,Cai_Wu_2014,DonohoJin2004,IngsterETAL2010,MukherjeePillaiLin2015}, it is less known if this optimality also holds under more general model assumptions for the whole family. In this paper we want to give a positive answer to this uncertainty. \\
As already mentioned, we differ between dense and sparse signals, where the main focus in the literature lies on the latter one. There are only a few positive results about the higher criticism test for dense signals. E.g., \citet{CaiJengJin2011} proved the optimality of the higher criticism test in the dense signal case for the normal location model introduced above with $\vartheta_n\to 0$. We will extend these results to general exponential families and to the whole family of \citet{JagerWellner2007}.\\
When we explained the results of \citet{Ingster1997} we omitted the case $r=\rho(\beta)$, the behavior on the detection boundary. \citet{Ingster1997} determined the limit distribution of the likelihood ratio test on the boundary under $\mathcal H_{0,n}$ as well as under $\mathcal H_{1,n}$. An interesting observation is that non-Gaussian limits do also occur. In other words, he showed that there is a third area in the $r$-$\beta$-parametrization plane, namely the nontrivial power area on the boundary. \citet{DitzhausJanssen2017} studied the asymptotic behavior of the likelihood ratio test and the higher criticism test on the detection boundary for general mixtures. In particular, they showed that the higher criticism test has no power on the boundary for various models, whereas the likelihood ratio test has nontrivial power there. We will extend this negative result to the whole family of \citet{JagerWellner2007}. \\
In short, this paper gives the following new insights into the tests based on $\Phi$-divergences introduced by \citet{Cai_Wu_2014}:
\begin{enumerate}[(i)]
	\item In contrast to \citet{JagerWellner2007} we do not restrict the parameter $s$ to the interval $[-1,2]$ and consider the family of test statistics $\{S_n(s):s\in\R\}$ instead. In particular, we extend the result of \citet{JagerWellner2007} about the asymptotic behavior of $S_n(s)$ under the null to the parameters $s\notin[-1,2]$.  
	
	\item The optimality of tests based on $S_n(s)$ for some $s\in\R$ even holds beyond the assumption of normality. In particular, we verify the optimality for the model recently suggested by \citet{Cai_Wu_2014} and for general dense mixtures based on exponential families. 
	
	\item On the detection boundary tests of the form $\varphi_n=\mathbf{1}\{S_n(s)>c_n(s)\}$ for some $s\in\R$ have no power asymptotically, whereas the likelihood ratio test does.
\end{enumerate} 
The paper is organized as follows. In \Cref{sec:teststat} we introduce the family of test statistics $\{S_n(s):s\in\R\}$ and present the limit distribution under the null $\mathcal H_{0,n}$, which is the same for the whole family. \Cref{sec:alt} contains our tools to discuss the asymptotic power of the whole family under the alternative $\mathcal H_{1,n}$. These tools are applied in \Cref{sec:appl} to a sparse signal model class recently suggested by \citet{Cai_Wu_2014} and a dense signal model based on general exponential families.
\section{The test family and its limit under the null}\label{sec:teststat}

\subsection{The test statistics}
This papers' focus lies on continuous noise distributions. That is why we can assume without loss of generality that $P_n=P_0$ for all $n$, e.g., having a transformation to $p$-values in mind. Denote the distribution function of $P_0$ by $F_0$. The basic idea is to compare the empirical distribution function $\mathbb{F}_n(u)=n^{-1}\sum_{i=1}^{n}\mathbf{1}\{X_{n,i}\leq u\}$ with the noise distribution function $F_0(u)$ for $u\in(0,1)$ by using one of the $\Phi$-divergences tests proposed by \citet{Csiszar1967} based on a convex function $\Phi$, see also \cite{AliSilvey1966,CressieRead1984,Csiszar1967}. To be more specific, we introduce a family $(\phi_s)_{s\in\R}$ of convex functions mapping $[0,\infty)$ to $\R\cup\{\infty\}$:
\begin{align*}
	\phi_s(x) = \begin{cases}
					x-\log (x)-1 &\text{ for }s= 0, \\
					x(\log (x) - 1)+1 &\text{ for }s= 1, \\
					 (1-s+sx-x^s)/(s(1-s)) &\text{ for }s\neq 0,1 .
				  \end{cases}
\end{align*}
 Based on these the family of $\Phi$-divergences statistics $(K_s)_{s\in\R}$ is given by
\begin{align*}
	K_s(u,v)= v \phi_s\Bigl(\frac{u}{v}\Bigr) + (1-v) \phi_s\Bigl(\frac{1-u}{1-v}\Bigr)
\end{align*}
for $u,v\in(0,1)$. It is easy to see that $\R\ni s\mapsto \phi_s(x)$ is continuous for every fixed $x\in(0,\infty)$ and so is $\R\ni s\mapsto K_s(u,v)$ for all fixed $u,v\in(0,1)$. Now, we consider the following family $\{S_n(s):s\in\R\}$ of test statistics for \eqref{eqn:testing_problem} given by
\begin{align}\label{eqn:test_statistic}
	S_n(s) = \sup_{X_{1:n}\leq x< X_{n:n}} K_s(\mathbb{F}_n(x),F_0(x)),
\end{align}
where $X_{1:n}\leq X_{2:n}\leq \ldots\leq X_{n:n}$ denote the order statistics of the observation vector $(X_{n,1},\ldots,X_{n,n})$. As explained by \citet{JagerWellner2007}  Tukey's higher criticism test $(s=2)$, the test of \citet{BerkJones1979} $(s=1)$, the "reversed Berk-Jones" statistic introduced by \citet{JagerWellner2004} $(s=0)$ and a studentized version of the higher criticism statistic studied by \citet{Eicker1979} $(s=-1)$ are included in this family. Note that $S_n(s)$ does not always coincide with the corresponding known test statistic but is equivalent to them for $s$ given in the parenthesis. For all other $s$ the test statistic $S_n(s)$ was new. \citet{JagerWellner2007} give a special emphasis to  $S_n(1/2)$, which is equivalent to the supremum of the pointwise Hellinger divergences between two Bernoulli distributions with parameters $F_0(u)$ and $\mathbb{F}_n(u)$, $u\in(0,1)$. 

\subsection{Limit distribution under the null}
The limit distribution of the higher criticism statistic is already known, see \citet{Jaeschke1979} and also Section 16.1 of \citet{ShorackWellner1986}, and so one can derive the limit distribution of $S_n(2)$. \citet{JagerWellner2007} showed that $nS_n(s)-nS_n(2)$ converges in probability to $0$ under the null $\mathcal H_{0,n}$ and, consequently, the limit distribution is the same for all $s\in[-1,2]$. We now extend their result to all $s\in\R$. \\
\textbf{Notation for convergences:} By $\overset{\mathrm d}{\longrightarrow},\overset{\mathrm P_0^n}{\longrightarrow},\overset{\mathrm Q_n^n}{\longrightarrow}$ we denote convergence in distribution, in $P_0^n$-probability and in $Q_n^n$-probability, respectively. 

\begin{theorem}\label{theo:null}
	Define
	\begin{align*}
		r_n= \log\log(n) + \frac{1}{2} \log\log\log(n) - \frac{1}{2} \log(4\pi).
	\end{align*}
	Then we have for all $s\in\R$ that under the null $\mathcal H_{0,n}$
	\begin{align}\label{eqn:limit_null}
		nS_n(s) - r_n \overset{\mathrm d}{\longrightarrow} Y,
	\end{align}	
	where $Y+\log(4)$ is standard Gumbel distributed, i.e. $x\mapsto \exp(-4\exp(-x))$ is the distribution function of $Y$.
\end{theorem}
At least for $S_n(2)$ it is known that the convergence rate is really slow, see \citet{KhmaladzeShin2001}. Since the basic idea of the proof of \Cref{theo:null} is to approximate $nS_n(s)$ by $nS_n(2)$ the same bad rate can be expected for all $s\in\R$ or an even worse rate due to an additional approximation error. Consequently, we cannot recommend using a critical value based on the convergence result in \Cref{theo:null} unless the sample size $n$ is really huge. Since the noise distribution is assumed to be known, a possibility to estimate the $\alpha$-quantile of $S_n(s)$ is to use a standard Monte-Carlo Simulation. Alternatively, you can find finite recursion formulas for the exact finite distribution in the literature, see \citet{JagerWellner2004} (for $s=0$, up to $n=1000$) and \citet{KhmaladzeShin2001} (for $s=2$, up to $n=10^4$). 

\section{Asymptotic power under the alternative}\label{sec:alt}

In this section we present a tool to analyze the asymptotic power behavior of all family members under the alternative $\mathcal H_{n,1}$. This was already done by \citet{DitzhausJanssen2017}  for the higher criticism test, i.e. for $s=2$. We show that the main tool to obtain their results can be used for general $s\in\R$. To be more specific, this tool is the following function $H_n:(0,1/2)\to (0,\infty)$ given by
\begin{align}\label{eqn:def_H}
	H_n(v)=  \frac{\sqrt{n} \varepsilon_{n}}{ \sqrt{v}} \Bigl(  | \mu_n^{F_0}(0,v]-v| + |\mu_n^{F_0}(1-v,1)-v| \Bigr), \, v\in\Bigl(0,\frac{1}{2}\Bigr),
\end{align}
where $\mu_n^{F_0}$ is the distribution of $F_0(X_{n,1})$ if $X_{n,1}\sim \mu_n$. The basic idea is to compare the tails near to $0$ as well as near to $1$ of the $p$-value $F_0(X_{n,1})$ if $X_{n,1}$ follows the signal distribution $\mu_n$ and the noise distribution $P_0$, respectively. Due to symmetry $H_n(v)$ does not change if we consider the $p$-value $1-F_0(X_{n,1})$ instead. Moreover, due to this it is sufficient to consider $v\in(0,1/2)$ instead of $v\in(0,1)$.  
\begin{theorem}[Complete detection]\label{theo:full_power}
	Suppose that there is a sequence $(v_n)_{n\in\N}$ in $(0,1/2)$ such that $v_nn\to \infty$ and $(\log\log(n))^{-1}H_n(v_n)\to\infty$. 
	Then we have for all $s\in\R$  
	\begin{align}\label{eqn:full_power_result}
		nS_n(s)-r_n \longrightarrow \infty \text{ in }Q_{n}^n\text{-probability}.
	\end{align}
\end{theorem}
Under \eqref{eqn:full_power_result} there exists a sequence of critical values $c_{n}(s)$ for all $s\in\R$ such that the sum of type 1 and 2 error probabilities of the test $\varphi_{n}(s)= \mathbf{1}\{S_n(s)>c_{n}(s) \}$ tends to $0$. In other words, by using $S_n(s)$ we can completely separate $\mathcal H_{0,n}$ and $\mathcal H_{1,n}$ asymptotically. \\
For the next result recall that   the null (product) distribution $P_0^n$ and the alternative (product) distribution $Q_n^n$ are said to be mutually contiguous if for all sequences $(A_n)_{n\in\N}$ of sets: $P_0^n(A_n)\to 0$ if and only if $Q_n^n(A_n)\to 0$. By the first Lemma of Le Cam $P_0^n$ and $Q_n^n$ are mutually contiguous if and only if the limits (in distribution) of the likelihood ratio test statistic are real-valued under the null $P_0^n$ as well as under the alternative $Q_n^n$. This typically holds on the detection boundary, see \citet{DitzhausJanssen2017}. Hence, mutual contiguity is no real restriction when discussing the behavior on this boundary. 
\begin{theorem}[No power]\label{theo:undetect}
	Suppose that $P_0^n$ and $Q_n^n$ are mutual contiguous and that there are constants $\kappa,c_{1,n},c_{2,n},c_{3,n},c_{4,n}\in(0,1)$ such that
	\begin{align}
		&\sqrt{\log\log(n)}\sup\{ H_n(v):v\in[c_{1,n},c_{2,n}]\cup[c_{3,n},c_{4,n}]\}\to 0,\textrm{ where }\label{eqn:undetect_Hn}\\
		&\frac{\log( c_{1,n})}{\log(n)}\to -1, \;\frac{\log(c_{4,n}) }{\log(n)}\to 0,\; \textrm{ and }\frac{\log (c_{2,n})}{\log(n)},\frac{\log(c_{3,n})}{\log(n)}\to -\kappa.\label{eqn:undetect_constants} 
	\end{align}  
	Then \eqref{eqn:limit_null} also holds under the alternative $\mathcal H_{n,1}$. 
\end{theorem}
Under \eqref{eqn:undetect_Hn} and \eqref{eqn:undetect_constants} all tests of the form  $\varphi_{n}(s)= \mathbf{1}\{S_n(s)>c_{n}(s) \}$ cannot distinguish between $\mathcal H_{0,n}$ and $\mathcal H_{1,n}$ asymptotically, i.e. the sum of type 1 and 2 error probabilities tends to $1$.
As we will prove, the supremum in \eqref{eqn:test_statistic} is not taken by values of $x$ from the extreme tails or from the middle with probability tending to $1$, in other words for $x$ with $F_0(x),1-F_0(x)\notin C_n=[c_{1,n},c_{2,n}]\cup[c_{3,n},c_{4,n}]$. To be more specific, we verify that under $\mathcal H_{0,n}$ (and so under $\mathcal H_{1,n}$ if $P_0^n$ and $Q_n^n$ are mutually contiguous)
\begin{align*}
	n\Bigl(  \sup_{X_{1:n}\leq x < X_{n:n}: F_0(x),1-F_0(x)\notin C_n} K_s(\mathbb{F}_n(x),F_0(x))\Bigr)-r_n\to -\infty.
\end{align*}
This briefly explains why in \eqref{eqn:undetect_Hn} we take the supremum over $C_n$ instead over $(0,1/2)$. \\

\begin{remark}[Simplification for the sparse case]\label{rem:simp_Hn}
	Typically, in the sparse case we even have $\sqrt{\log\log(n)}n\varepsilon_{n}^2\to 0$. Then it is easy to see that the statements in \Cref{theo:full_power,theo:undetect} remain true if $H_n(v)$ is replaced by
		\begin{align*}
			\widetilde H_n(v)= \sqrt{n} \varepsilon_{n} v^{-\frac{1}{2}} \Bigl(  \mu_n(0,v] + \mu_n[1-v,1) \Bigr).
		\end{align*}
\end{remark}

\section{Application}\label{sec:appl}

\subsection{Extension of \citet{Cai_Wu_2014}}
Throughout this section we consider (only) the sparse case
\begin{align*}
	\varepsilon_{n}=n^{-\beta}, \,\beta\in\Bigl(\frac{1}{2},1\Bigr].
\end{align*}
Starting with a fixed noise distribution $P_0$ and a fixed sequence $(\mu_n)_{n\in\N}$ of signal distributions, \citet{Cai_Wu_2014} developed a technique to calculate a detection threshold $\beta^{\#}$ for the parameter $\beta$.
\begin{enumerate}[(i)]
	\item (Undetectable) \label{enu:theo:cai_wu_uninfo}  If $\beta$ exceeds $\beta^\#$ then there is no sequence of tests, which can distinguish between $\mathcal H_{0,n}$ and $\mathcal H_{1,n}$ asymptotically.
	
	\item (Completely detectable) \label{enu:theo:cai_wu_fullinfo} If $\beta$ is smaller than $\beta^\#$ then there is a sequence of likelihood ratio tests, which can completely separate $\mathcal H_{0,n}$ and $\mathcal H_{1,n}$ asymptotically.
\end{enumerate}
Let us begin this section by recalling the results of \citet{Cai_Wu_2014}. We first present the special case of standard normal noise.
\begin{theorem}[Theorem 1 and 4 in \cite{Cai_Wu_2014}] \label{theo:Cai+Wu_norm}
	Define for all $x>0$
	\begin{align*}
		\widetilde h_{n}(x)= \log\Bigl( \frac{\mathrm{ d } \mu_{n}}{\,\mathrm{ d }P_{0}}( x \sqrt{2\log(n)}) \Bigr). 
	\end{align*}
	Suppose that 
	\begin{align*}
		\sup \Bigl\{ \Bigl | \frac{\widetilde h_n(x)}{\log(n)}-\alpha(x) \Bigr| :x\in\R\Bigr\} \to 0
	\end{align*}
	for a measurable $\alpha:\R\to\R$. Then the detection boundary is given by
	\begin{align*}
		\beta^{\#}=\frac{1}{2}+  \essup_{t\geq 0} \Bigl\{ \alpha(t)-t^2+ \frac{\min(t^2,1)}{2} \Bigr\},
	\end{align*}
	where $\essup_{x\geq 0}{g(x)}=\inf\{ K\in\R: \lebesgue( g \geq K) = 0\}$ denotes the essential supremum of a measurable function $g:[0,\infty)\to\R$. Moreover, if $\beta<\beta^{\#}$ then there is a sequence $(c_n)_{n\in\N}$ of critical values such that $\varphi_n=\mathbf{1}\{S_n(1/2)>c_n\}$ (higher criticism test) can completely separate $\mathcal H_{0,n}$ and $\mathcal H_{1,n}$ asymptotically. 
\end{theorem}
The results concerning the normal location model mentioned in \Cref{sec:intro} follow immediately from this Theorem, for details see Section V-A and V-C in \cite{Cai_Wu_2014}. More generally, \Cref{theo:Cai+Wu_norm} can be applied to the model given by $P_0=N(0,1)$ and $\mu_n=\widetilde{\mu}_n*N(0,1)$, where $*$ denotes the convolution. For details we refer the reader to Corollary 1 and Section V-B in \cite{Cai_Wu_2014}. An example for this convolution idea is the heteroscedastic normal location model, where the variance of the signal distribution $\mu_n=N(\vartheta_n,\sigma_0^2)$ may differ from 1, i.e. $\sigma_0^2\neq 1$ is allowed. Note that \citet{CaiJengJin2011} already studied the optimality for the higher criticism test under this model.  \\
For non-normal noise distributions $P_0$ the following theorem can be applied:

\begin{theorem}[Theorem 3 in \cite{Cai_Wu_2014}] \label{theo:Cai+Wu}
	Define for all $t>0$
	\begin{align*}
		&h_{n,1}(t)= \log \Bigl( \frac{\mathrm{ d } \mu_{n}}{\,\mathrm{ d }P_{0}}\bigl(F_{0}^{-1}\bigl( n^{-t} \bigr) \bigr) \Bigr),\, 
		h_{n,2}(t)=\log\Bigl(  \frac{\mathrm{ d } \mu_{n}}{\,\mathrm{ d }P_{0}}\bigl(F_{0}^{-1}\bigl(1- n^{-t} \bigr) \bigr) \Bigr) \\
		&\textrm{and }\, h_{n}(t)= \max\bigl\{ h_{n,1}(t), h_{n,2}(t) \bigr\}. 	
	\end{align*}
	Suppose that 
	\begin{align}\label{eqn:theo:Cai+Wu_con}
		 \sup \Bigl\{ \Bigl | \frac{h_n(t)}{\log(n)}-\gamma(t) \Bigr| :t \geq \frac{\log(2)}{\log(n)}\Bigr\} \to 0
	\end{align}
	for a measurable $\gamma:[0,\infty)\to\R$. Then the detection threshold for $\beta$ is given by
	\begin{align}\label{eqn:cai_wu_beta}
		\beta^{\#}=\frac{1}{2}+  \essup_{t\geq 0} \Bigl\{ \gamma(t)-t+ \frac{\min(t,1)}{2} \Bigr\}.
	\end{align}
\end{theorem}
Theorem \ref{theo:Cai+Wu} can be applied to derive the detection boundary for the general Gaussian location mixture model and the exponential-$\chi^2$ mixture model as explained by \citet{Cai_Wu_2014}. Note that \citet{DonohoJin2004} already discussed these models and postulated the optimality of the higher criticism test for them. But it was not known if this optimality holds in general under the assumptions of \Cref{theo:Cai+Wu}. According to the next theorem it does, even for the whole family of test statistics $\{S_n(s):s\in\R\}$. 
\begin{theorem}[Extension of \Cref{theo:Cai+Wu}]\label{theo:cai_wu_ext} Let $h_n$ be defined as in \Cref{theo:Cai+Wu}. Assume that there exists some $\beta^*\in\R$ such that for every $\delta>0$
	\begin{align}
		\liminf_{n\to\infty}  \lebesgue\left(  t\geq \frac{\log(2)}{\log(n)}: \beta^*-\delta -\frac{1}{2} \leq \frac{h_{n}(t)}{\log(n)} - t + \frac{\min\{t,1\}}{2} \right) > 0 \label{eqn:cor:cai_wu_ext_Leb_>_0} \\
		\text{ and } \lebesgue \left(  t\geq \frac{\log(2)}{\log(n)}: \beta^*+\delta -\frac{1}{2} \leq  \frac{h_{n}(t)}{\log(n)} - t + 	\frac{\min\{t,1\}}{2} \right) =0 \label{eqn:cor:cai_wu_ext_Leb_to_0} 
	\end{align}
	for all sufficiently large $n\geq N_{1,\delta}$. Let $(\lambda_n)_{n\in\N}$ be a sequence in $(0,\infty)$ such that $\lambda_n\to 0$ and $\lambda_n n^{\kappa}\to \infty$ for all $\kappa>0$. Suppose that for some $M\geq 1$:
	\begin{align}
		& \lim_{n\to\infty}    \sup_{t \geq M }  \left|  \frac{h_{n}(t)}{\log(n)} - \gamma(t) \right| =0 \label{eqn:cor:cai_wu_ext_techn_cond_uniform}\\
		\intertext{for some $\gamma:(0,\infty)\to\R$ \underline{or} for every $\delta>0$ there exists $N_{2,\delta}\in\N$ such that }
		& \lebesgue \left(  t\geq M: \beta^*+\delta -1 \leq  \frac{h_{n}(t)}{\log(n)} - \left( 1- \frac{\lambda_n}{\log(n)} \right) t \right)  = 0 \label{eqn:cor:cai_wu_ext_techn_cond_lebes} 	
	\end{align}
	for all $n\geq N_{2,\delta}$. Then $\beta^{\#}=\beta^*$. Moreover, if $\beta<\beta^{\#}$ then for every $s\in\R$ there is a sequence $(c_n(s))_{n\in\N}$ of critical values such that $\varphi_n=\mathbf{1}\{S_n(s)>c_n(s)\}$ can completely separate $\mathcal H_{0,n}$ and $\mathcal H_{1,n}$ asymptotically. 
\end{theorem}
The conditions \eqref{eqn:cor:cai_wu_ext_Leb_>_0} and \eqref{eqn:cor:cai_wu_ext_Leb_to_0} together are mimicking the essential supremum in \eqref{eqn:cai_wu_beta}, where $\gamma$ is replaced by $h_n/\log(n)$. The advantage is that we do not need the uniform convergence as in  \eqref{eqn:theo:Cai+Wu_con}. During our study we had a look at a simple scale exponential distribution model, for which we expected that \Cref{theo:Cai+Wu} can be applied. But for this model $h_n(t)/\log(n)$ tends to $-\infty$ for small $t>0$ and so \eqref{eqn:theo:Cai+Wu_con} is violated. Our (new) more general assumptions can handle this problem. Furthermore, we can treat the following exponential families including a scale exponential, a scale Fr\'echet and a location Gumbel distribution model. Before we can formulate the theorem let us recall that $L$ is called \textit{slowly varying} at infinity if $\lim_{x\to\infty}L(\lambda x)/L(x)= 1$ for $\lambda>0$.

\begin{theorem}\label{theo:cai_wu_expfam}
	Let $(P_{(\vartheta)})_{\vartheta\in[0,\infty)}$ be a family of continuous distributions on $[0,\infty)$ with $P_{(\vartheta)}\ll P_{(0)}$ and Radon-Nikodym density
	\begin{align}\label{eqn:expon_fam_density}
		\frac{\mathrm{ d } P_{(\vartheta)}}{\,\mathrm{ d }P_{(0)}} = C({\vartheta}) \exp \left( {\vartheta}{T} \right)
	\end{align}
	for appropriate functions ${T}:[0,\infty)\to \R$ and $C:[0,\infty)\to (0,\infty)$ with $C(0)=1$. Suppose that $T$ is strictly decreasing on $[0,\eta]$ for some $\eta>0$, $T(\eta)\geq T(x)$ for all $x\geq \eta$ and 
	\begin{align*}
		T(F_0^{-1}(0))-T( F_0^{-1})(u)=u^{\frac{1}{p}}L\Bigl( \frac{1}{u} \Bigr) \text{ as }u\searrow 0
	\end{align*}
	for a slowly varying function $L$ at infinity, where $F_0^{-1}$ is the left continuous quantile function of $P_{(0)}$. Let $P_0=P_{(0)}$ be the noise distribution and $\mu_n=P_{(\vartheta_n)}$ the signal distribution for $\vartheta_n=n^{r}$ with $r>0$.
	Then the conditions of \Cref{theo:cai_wu_ext} are fulfilled for
	\begin{align*}
		\beta^\#=\beta^\#(r,p)= \frac{\min\{rp,1\}+1}{2}.
	\end{align*}
\end{theorem}
Due to lack of space we do not discuss the asymptotic behavior of the tests on the detection boundary, i.e. $\beta=\beta^\#$. We refer the reader to Corollary 5.7 and Theorem 8.19 of \citet{Ditzhaus2017}. In short, the likelihood ratio test has nontrivial power on the detection boundary, whereas the higher criticism test has no power. By \Cref{theo:undetect} the latter can be extended to the whole family $\{S_n(s):s\in\R\}$. \\
At the end of this section we present the extension of \Cref{theo:Cai+Wu_norm}. This can be applied, e.g., to verify the optimality of the whole family $\{S_n(s):s\in\R\}$ for the heteroscedastic $(\sigma_0^2\neq 1)$ and heterogeneous ($\sigma_0^2=1$) normal location model with $\mu_n=N(\vartheta_n,\sigma_0^2)$. The behavior of the likelihood ratio test and the higher criticism test on the detection boundary for this particular model is discussed in \citet{DitzhausJanssen2017} and can be extended, using \Cref{theo:full_power,theo:undetect}, to the whole family $\{S_n(s):s\in\R\}$. Again, the likelihood ratio test has nontrivial power, whereas tests of the form $\varphi_n=\mathbf{1}\{S_n(s)>c_n(s)\}$ for some $s\in\R$ have no asymptotic power.
\begin{theorem}[Extension of  \Cref{theo:Cai+Wu_norm}]\label{theo:cai_wu_ext_norm} Let $\widetilde h_n$ be defined as in \Cref{theo:Cai+Wu_norm}. Suppose that there is some $\beta^*\in \R$ such that for every $\delta>0$
	\begin{align}
		\liminf_{n\to\infty}  \lebesgue\Bigl(  x\in\R: \beta^*-\delta -\frac{1}{2} \leq \frac{\widetilde h_{n}(x)}{\log(n)} - x^2 + \frac{\min\{x^2,1\}}{2} \Bigr) > 0 \label{eqn:cor:cai_wu_norm_Leb_>_0} \\
		\text{ and } \lebesgue \Bigl(  x\in\R: \beta^*+\delta -\frac{1}{2} \leq  \frac{\widetilde h_{n}(x)}{\log(n)} - x^2 + \frac{\min\{x^2,1\}}{2} \Bigr) =0 \label{eqn:cor:cai_wu_norm_Leb_to_0} 
	\end{align}
	for all sufficiently large $n\geq N_{1,\delta}$. Let $(\lambda_n)_{n\in\N}$ be a sequence in $(0,\infty)$ such that $\lambda_n\to 0$ and $\lambda_nn^{\kappa}\to \infty$ for all $\kappa>0$. Suppose that for some $M\geq 1$:
	\begin{align}
		& \lim_{n\to\infty}    \sup_{|x| \geq M }  \Bigl|  \frac{\widetilde h_{n}(x)}{\log(n)} - \alpha(x) \Bigr| =0 \label{eqn:cor:cai_wu_norm_techn_cond_uniform}\\
		\intertext{for some $\alpha:(0,\infty)\to\R$ \underline{or} for every $\delta>0$ there exists $N_{2,\delta}\in\N$ such that }
		& \lebesgue \Bigl(  |x|\geq M: \beta^*+\delta -1 \leq  \frac{\widetilde h_{n}(x)}{\log(n)} - \left( 1- \frac{\lambda_n}{\log(n)}  \right) x^2 \Bigr)  = 0 \label{eqn:cor:cai_wu_norm_techn_cond_lebes} 	
	\end{align}
	for all $n\geq N_{2,\delta}$. Then $\beta^{\#}=\beta^*$. Moreover, if $\beta<\beta^{\#}$ then for every $s\in\R$ there is a sequence $(c_n(s))_{n\in\N}$ of critical values such that $\varphi_n=\mathbf{1}\{S_n(s)>c_n(s)\}$ can completely separate $\mathcal H_{0,n}$ and $\mathcal H_{1,n}$ asymptotically. 
\end{theorem}

\subsection{Dense exponential family}
In this section we give a quite general example for a dense signal model and show the optimality of the whole family $\{S_n(s):s\in\R\}$ for it.  The normal location model, $P_0=N(0,1)$ and $\mu_n=N(\vartheta_n,1)$ with $\vartheta_n \to 0$, is included in this consideration. This particular model was already discussed by \citet{CaiJengJin2011} concerning the higher criticism test, i.e. for $s\in\R$. Here, we extend their result to all $s\in\R$ and to exponential families in a far more general form compared to \Cref{theo:cai_wu_ext}. In contrast to the previous section we discuss the asymptotic power behavior on the detection boundary here in detail. 

\begin{figure}[tb] 
	\begin{center}	
		\includegraphics[trim = 22mm 150mm 80mm 10mm, clip, width=0.45\textwidth]{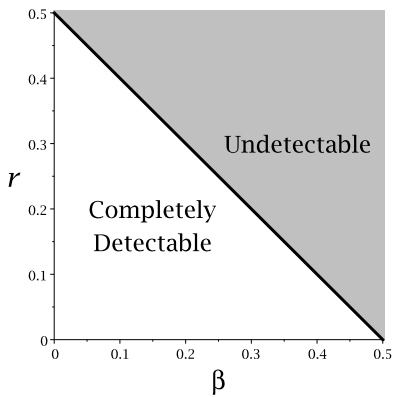}
	\end{center}
	\caption[Detection boundary for dense exponential families]{
		The detection boundary $\beta\mapsto \rho^*(\beta)= 1/2-\beta$ for dense signal exponential family mixtures  is plotted, see \Cref{theo:den_exp}.  } \label{fig:normal+expfam}
\end{figure}

\begin{theorem}\label{theo:den_exp}
	Let $(P_{(\vartheta)})_{\vartheta\in[0,\infty)}$ be a family of continuous distributions on $[0,\infty)$ with Radon-Nikodym density given by \eqref{eqn:expon_fam_density} for ${T}:[0,\infty)\to \R$ and $C:[0,\infty)\to (0,\infty)$ with $C(0)=1$. Assume that $\text{Var}_{P_{(0)}}(T)>0$. Consider the noise distribution $P_0=P_{(0)}$ and the signal distribution $\mu_n=P_{(\vartheta_n)}$ with the parametrization
	\begin{align*}
		\varepsilon_{n}=n^{-\beta},\,\beta\in\Bigl( 0,\frac{1}{2} \Bigr),\textrm{ and }\vartheta_n=n^{-r},\,r>0,
	\end{align*}
	then the detection boundary for the parameter $r$ is given by
	\begin{align*}
		\rho^{*}(\beta)=\frac{1}{2}-\beta.			
	\end{align*}
	In particular, we have for all $s\in\R$:
	\begin{enumerate}[(a)]
		\item\label{enu:theo:den_exp_r<} If $r<\rho^{*}(\beta)$ then there is a sequence $(c_n(s))_{n\in\N}$ of critical values such that $\varphi_n=\mathbf{1}\{S_n(s)>c_n(s)\}$ can completely separate $P_0^n$ and $Q_n^n$ asymptotically. 
		
		\item\label{enu:theo:den_exp_r=} Suppose that $r=\rho^{*}(\beta)$. Then  $2N(0,1)((1/2){\text{Var}_{P_{(0)}}(T)}^{\frac{1}{2}},\infty)$ is the sharp lower bound of the limit of the sum of type 1 and 2 error probabilities for all tests testing $P_0^n$ versus $Q_n^n$. But all tests of the form $\varphi_n=\mathbf{1}\{S_n(s)>c_n(s)\}$ cannot distinguish between $P_0^n$ and $Q_n^n$ asymptotically.
		
		\item\label{enu:theo:den_exp_r>} If $r>\rho^{*}(\beta)$ then no test $\varphi_n$ can distinguish between $P_0^n$ and $Q_n^n$ asymptotically.
	\end{enumerate}
\end{theorem}

\section{Discussion}
	The higher criticism test statistic became quite popular recently. In this paper we showed that a whole class of test statistics shares the same optimal properties of the higher criticism one under different model assumptions. The advantage of a whole class is more flexibility in choosing a test statistic which suits the specific problem best. \citet{JagerWellner2007} already pointed out that $S_n(s)$ is more sensible for signal distributions with heavy or light tails if $s\leq 0$ or $s\geq 1$, respectively. As a good compromise they suggested their "new" $S_n(1/2)$. In future we wish to conduct detailed simulation study in order to understand the differences between the test statistics and to be able to give an advice for practitioners how to choose "the best" $s$. \\
	Beside the detection problem, a more in-depth analysis of the data as feature selection, estimation of the number of signals and classification is of great interest to practice, too. The detection problem discussed in this paper is closely related to the other problems, see \cite{DonohoJin2009,DonohoJin2015,HallPittelkowGhosh2008,Jin2009}, for which the higher criticism statistic can be used, too. The results in this paper suggest that the whole class $\{S_n(s):s\in\R\}$ may be used for these problems. The benefit would again be more flexibility. This could be a project for the future.\\    
	The function $H_n$ serving as our tool for the asymptotic behavior under the alternative was already used by \citet{DitzhausJanssen2017} for the higher criticism test. In particular, the results concerning their examples can be extended immediately to the whole family $\{ S_n(s):s\in\R\}$. Beside the normal location model they also discussed a structure model for $p$-values based on the spike chimeric alternatives of \citet{Khmaladze1998}.
    
\section{Acknowledgments}
	The author thanks the \textit{Deutsche Forschungsgemeinschaft} (DFG) for financial support (DFG Grant no. 618886).
	    
\section{Proofs}\label{sec:proof}
To prove \Cref{theo:null} we use some results of \citet{Chang1955} and \citet{Wellner1978} about the asymptotic behavior of the empirical distribution function. We summarize them in the following lemma.
\begin{lemma}\label{lem:Chang+Wellner}
	Let $X_{n,1},\ldots,X_{n,n}$ be independent and identical distributed random variables on the same probability space $\left(\Omega, \mathcal{A}, {P}\right)$ with continuous distribution function $F_n$. Let $\mathbb{F}_n$ be the corresponding empirical distribution function. Let $(d_n)_{n\in\N}$ be a decreasing sequence in $\R$, i.e. $d_n>d_{n+1}$, such that $nF_n(d_n)\to\infty$. Then
	\begin{align}\label{eqn:lem:Chang+Wellner}
		&\sup_{d_n\leq x< \infty}\Bigl|\frac{\mathbb{F}_n(x)}{F_n(x)}-1\Bigr| \overset{P}{\longrightarrow} \,0. 
	\end{align}
	If additionally $c_n= nF(d_n)/\log\log(n)\to\infty$ then
	\begin{align}\label{eqn:lem:Chang+Wellner_sqrt2}
		\sqrt{c_n}\sup_{d_n\leq x <\infty}\Bigl|\frac{\mathbb{F}_n(x)}{F_n(x)}-1\Bigr| \overset{P}{\longrightarrow} \sqrt{2}.
	\end{align}
	Moreover, for all $t\in\R$
	\begin{align}\label{eqn:lem:Wellner}
		\lim_{\lambda\to\infty}\limsup_{n\to\infty}P( B_{n,\lambda,t}^c)  =0 \text{ with }B_{n,\lambda,t}=\Bigl\{\sup_{  x \in(X_{1:n},\infty)}\Bigl(  \frac{ \mathbb{F}_n(x) }{x}  \Bigr)^t< \lambda \Bigr\}
	\end{align}
\end{lemma}
\begin{proof}[Proof of \Cref{lem:Chang+Wellner}]
	First, suppose that $X_{n,1},\ldots,X_{n,n}$ are uniformly distributed on $(0,1)$. Then \eqref{eqn:lem:Chang+Wellner} was stated by \citet{Chang1955}, see also Theorem 0 in \cite{Wellner1978}, and \eqref{eqn:lem:Wellner} follows by combining (i) and (ii) of  Remark 1 of \citet{Wellner1978}. Moreover, \eqref{eqn:lem:Chang+Wellner_sqrt2} follows from Theorem 1S of \citet{Wellner1978}. For general continuous distribution  note that $F_n(X_{n,1}),\ldots,F_n(X_{n,n})$ are independent and uniformly distributed random variables in $(0,1)$. Consequently, it is easy to check that the statements for general distributions can be concluded from the ones for uniform distributions.
\end{proof}

\begin{proof}[Proof of \Cref{theo:null}]
		Having a transformation to $p$-values $p_{n,i}=F_0(X_{n,i})$ or $p_{n,i}=1-F_0(X_{n,i})$ in mind we can assume without loss of generality that $P_0=\lebesgue_{|(0,1)}$ is the uniform distribution on the interval $(0,1)$. The proof  is based, as is the one of Theorem 3.1 of \citet{JagerWellner2007}, on a Taylor expansion of $u\mapsto K_s(u,v)$ around $u=v$. It is easy to verify that
	\begin{align*}
		&\frac{\partial}{\partial u}K_s(u,v)_{|u=v} = 0 = K_s(v,v),\,\quad \frac{\partial^2}{\partial^2 u}K_s(u,v)_{|u=v} = \frac{1}{v(1-v)}\\
		\text{and }&\frac{\partial^3}{\partial^3 u}K_s(u,v) = \frac{(s-2)}{v^2} \Bigl( \frac{u}{v} \Bigr)^{s-3}-\frac{(s-2)}{(1-v)^2} \Bigl( \frac{1-u}{1-v} \Bigr)^{s-3}.
	\end{align*}
	Hence, by a careful calculation we obtain for all $x\in[X_{1:n},X_{n:n})$ that
	\begin{align}
		& K_s(\mathbb{F}_n(x),x)=  K_2(\mathbb{F}_n(x),x) \Bigl(1+ \frac{(s-2)}{3}R_{n,x,s}\Bigr) \text{ with }\label{eqn:Ks()...taylor} \\
		&R_{n,x,s}=\frac{( \mathbb{F}_n(x)-x )}{x(1-x)} \Bigl( (1-x)^2 \Bigl( \frac{\mathbb{F}^*_{n,x}}{x} \Bigr)^{s-3}-x^2 \Bigl( \frac{1-\mathbb{F}_{n,x}^*}{1-x} \Bigr)^{s-3} \Bigr),\nonumber 
	\end{align}
	where $\mathbb{F}_{n,x}^*$ is a random variable satisfying $\min\{\mathbb{F}_n(x),x\}\leq \mathbb{F}_{n,x}^*\leq \max\{\mathbb{F}_n(x),x\}$. Clearly, $(0,\infty)\ni t\to t^{3-s}$ is monotone for all $s\in \R$. That is why we have $|R_{n,x,s}|\leq R_{n,x,s}^{(1)}+R_{n,x,s}^{(2)}$, where
	\begin{align*}
		 &R_{n,x,s}^{(1)}= \frac{|\mathbb{F}_n(x)-x|}{x}  \max\Bigl\{ 1,\Bigl( \frac{\mathbb{F}_n(x)}{x} \Bigr)^{s-3} \Bigr\}\\
		 \text{ and }\;&R_{n,x,s}^{(2)}= \frac{|\mathbb{F}_n(x)-x|}{1-x}  \max\Bigl\{ 1,\Bigl( \frac{1-\mathbb{F}_n(x)}{1-x} \Bigr)^{s-3} \Bigr\}.		 
	\end{align*}
	Let $d_n=n^{-1}(\log n)^5$. Obviously, $P_0^n(X_{1:n}> d_n)+P_0^n(X_{n:n}< 1-d_n)=2(1-d_n)^n\to 0$. Moreover, observe that by \eqref{eqn:Ks()...taylor}
	\begin{align*}
		&\Bigl | nK_s(\mathbb{F}_n(x),x)-r_n -(nK_2(\mathbb{F}_n(x),x)-r_n) \Bigr| \\
		&\leq \Bigl( R_{n,x,s}^{(1)}+R_{n,x,s}^{(2)} \Bigr)\Bigl( \Bigl | nK_2(\mathbb{F}_n(x),x)-r_n  \Bigr|+ r_n \Bigr).
	\end{align*}
	 Consequently, for \eqref{eqn:limit_null} it is sufficient to show that 
	\begin{align}
		& n \Bigl( \sup_{  x \in(d_n,  1- d_n) } K_2(\mathbb{F}_n(x),x) \Bigr)-r_n \overset{\mathrm d}{\longrightarrow} Y\text{ under }P_0^n,\label{eqn:K2}\\
		& r_n\sup_{  x \in(d_n,  1- d_n) } R_{n,x,s}^{(j)} \xrightarrow{P_0^n} 0 \text{ for }j=1,2\,\text{ and }\label{eqn:Rj}\\
		& I_{n,s}= n \Bigl( \sup_{  x \in(X_{1:n},d_n)\cup(1-d_n,X_{n:n}) } K_s(\mathbb{F}_n(x),x) \Bigr) -r_n\xrightarrow{P_0^n} -\infty .\label{eqn:Ks-rn_to_-infty}
	\end{align}
	Note that
	\begin{align*}
		I_{n,s}\leq r_n \Bigl( -1+\frac{n}{r_n}  \sup_{  x \in(X_{1:n},d_n)\cup(1-d_n,X_{n:n}) } K_2(\mathbb{F}_n(x),x)\Bigl( 1+R^{(1)}_{n,x,s} +R^{(2)}_{n,x,s} \Bigr)\Bigr).
	\end{align*}
	Hence, using the inequality
	\begin{align} 	\label{eqn:ineque_Blambda}
		P(A_n) \leq P(A_n\cap B_{n,\lambda})+P(B_{n,\lambda}^c)
	\end{align}
	with appropriate sets we can deduce \eqref{eqn:Ks-rn_to_-infty} from \eqref{eqn:Ks()...taylor} if 
	\begin{align}
		&\frac{n}{r_n}  \sup_{  x \in(X_{1:n},d_n)\cup(1-d_n,X_{n:n}) } K_2(\mathbb{F}_n(x),x) \xrightarrow{P_0^n} 0 \text{ and } \label{eqn:K2/rn_to_0} \\
		&\lim_{\lambda\to\infty}\limsup_{n\to\infty}P_0^n\bigl( (B_{n,\lambda}^{(j)})^c\bigr)=0\text{ for both }j\in\{1,2\},\label{eqn:Rj_lambda}\\
		&\text{ where }\; B_{n,\lambda}^{(j)} =\Bigl\{ \sup_{  x \in(X_{1:n},d_n)\cup(1-d_n,X_{n:n}) } R_{n,x,s}^{(j)} < \lambda  \Bigr\}. \nonumber 
	\end{align}
	Thus, it remains to verify \eqref{eqn:K2}, \eqref{eqn:Rj}, \eqref{eqn:K2/rn_to_0} and \eqref{eqn:Rj_lambda}. Note that by symmetry it is sufficient to show \eqref{eqn:Rj} and \eqref{eqn:Rj_lambda} for $j=1$. Using again \eqref{eqn:ineque_Blambda} we obtain \eqref{eqn:Rj} for $j=1$ from \eqref{eqn:lem:Chang+Wellner_sqrt2} and \eqref{eqn:lem:Wellner} setting $t=s-3$ since $c_n= nd_n/\log\log(n)=\log(n)^5/\log\log(n)\to\infty$ and $\sqrt{c_n}/r_n\to \infty$. 	Since 
	\begin{align*}
		R_{n,x,s}^{(1)}&\leq  1+  \frac{\mathbb{F}_n(x)}{x} +\Bigl( \frac{\mathbb{F}_n(x)}{x} \Bigr)^{s-2}+\Bigl( \frac{\mathbb{F}_n(x)}{x} \Bigr)^{s-3}
	\end{align*}
	we can conclude \eqref{eqn:Rj_lambda} for $j=1$ by applying \eqref{eqn:lem:Wellner} for $t=1,s-2,s-3$.\\
	In order to prove the remaining \eqref{eqn:K2} and  \eqref{eqn:K2/rn_to_0} we introduce the supremum statistic of the normalized uniform empirical process
	\begin{align}\label{eqn:def_Zn_process}
	 	 \mathbb{Z}_n(a,b) = \sup_{a< x < b} \sqrt{n} \frac{ |  \mathbb{F}_n(x)-x | }{\sqrt{x(1-x)}},\, a,b\in(0,1),
	\end{align}
	studied by \citet{Jaeschke1979}, see also Chap. 16 of \citet{ShorackWellner1986}. In particular, by (19), (20), (25), (26) and (g) in Sec. 16.1 from \citet{ShorackWellner1986}  and the symmetry $\mathbb{Z}_n(0,a)\overset{d}{=}\mathbb{Z}_n(1-a,1)$ we have	
	\begin{align}
		&\frac{\mathbb{Z}_n(d_n,1-d_n)}{b_n} \xrightarrow{P_0^n} 1 \text{ and }\,b_n \mathbb{Z}_n(d_n,1-d_n) -c_n \overset{\mathrm d}{\longrightarrow} Y \text{ under }P_0^n,	\label{eqn:jaeschke_results1}	\\
		&b_n^{-1}\mathbb{Z}_n(0,d_n)\xrightarrow{P_0^n}0\;\text{ and }b_n^{-1}\mathbb{Z}_n(1-d_n,1)\xrightarrow{P_0^n}0\label{eqn:jaeschke_results2}, \text{ where }\\
		&b_n=\sqrt{2\log\log(n)}\text{ and }c_n=b_n^2 + \frac{1}{2} \log\log\log(n) - \frac{1}{2} \log(4\pi). \label{eqn:def_bn_cn}
	\end{align}
	Moreover, observe that
	\begin{align}
		&n \Bigl( \sup_{  x \in(a, b) } K_2(\mathbb{F}_n(x),x) \Bigr)-r_n = \frac{1}{2}\mathbb{Z}_n(a,b)^2-r_n \label{eqn:Sn(2)...1} \\
		&=\frac{1}{2}(b_n\mathbb{Z}_n(a,b)-c_n)\Bigl( \frac{\mathbb{Z}_n(a,b)}{b_n}+\frac{c_n}{b_n^2} \Bigr) + \Bigl( \frac{1}{2}\frac{c_n^2}{b_n^2}-r_n \Bigr), \, a,b\in(0,1)\label{eqn:Sn(2)=....}.
	\end{align}
	Finally, \eqref{eqn:K2} follows  from \eqref{eqn:Sn(2)=....} and \eqref{eqn:jaeschke_results1}. We can conclude \eqref{eqn:K2/rn_to_0} from \eqref{eqn:jaeschke_results2} and \eqref{eqn:Sn(2)...1} since $b_n^2/r_n\to 2$. 
\end{proof}

\begin{proof}[Proof of \Cref{theo:full_power}]
	Similar to the previous proof  we can assume $P_0=\lebesgue_{|(0,1)}$. Set $l_n=\log\log(n)$. Due to symmetry it is sufficient to give the proof under the assumption 
	\begin{align}
		&n^{\frac{1}{2}} \varepsilon_{n} v_n^{-\frac{1}{2}}l_n^{-1}(\mu_n(0,v_n]-v_n) \to A \in\{-\infty,\infty\}	\label{eqn:fullpower_proof_ass1}.
	\end{align}
	Let $G_n$ be the distribution function of $Q_n$, i.e. $ G_n(v)=v+\varepsilon_{n}(\mu_n(0,v]-v)$ for all $v\in(0,1)$. 
	If $A=\infty$ then it is easy to see that 
	\begin{align}\label{eqn:fullpower_proof_nF(v_n)}
		l_n^{-1}nG_n(v_n) \to \infty.
	\end{align}
	Note that $A=-\infty$ implies $l_n^{-1}\sqrt{nv_n}\to\infty$ and so $l_n^{-1}nv_n(1-\varepsilon_{n})\to \infty$. Hence,  \eqref{eqn:fullpower_proof_nF(v_n)} holds in both cases for $A\in\{-\infty,\infty\}$.  Due to \eqref{eqn:fullpower_proof_nF(v_n)} we obtain from \Cref{lem:Chang+Wellner} that
	\begin{align}\label{eqn:fullpower_proof_Chang+Wellner}
		\frac{\mathbb{F}_n(v_n)}{G_n(v_n)} \to 1 \text{ in }Q_n^n\text{-probability}.
	\end{align}
	Observe that $Q_n^n(X_{1:n}<v_n<X_{n:n})\to 1$ and so 
	\begin{align*}
		Q_n^n\Bigl( nS_n(s) - r_n \geq l_n \Bigl( \frac{nv_n}{l_n} \Phi_s\Bigl( \frac{\mathbb{F}_n(v_n)}{v_n} \Bigr)  - \frac{r_n}{l_n}\Bigr) \Bigr)\to 1.
	\end{align*}
	Since $r_nl_n^{-1}\to 1$ for proving \eqref{eqn:full_power_result} it is sufficient to verify  that
	\begin{align}\label{eqn:fullpower_proof_suff}
		Q_n^n\Bigl( \,\frac{ nv_n}{l_n}\Phi_s\Bigl( \frac{\mathbb{F}_n(v_n)}{v_n} \Bigr) > 2 \,\Bigr) \to 1.
	\end{align}
	Since $v_n^{-1}G_n(v_n)\geq 1-\varepsilon_{n}\to 1$ we can assume without loss of generality that 
	\begin{align}\label{eqn:fullpower_proof_toC}
		v_n^{-1}G_n(v_n)\to C\in[1,\infty],
	\end{align}
	otherwise we use standard subsequence arguments.\\
	First, consider $C< \infty$. Since $\Phi_s(1)=\Phi_s(1) = 0$ and $\Phi''_s(x)=x^{s-2}$, $x>0$, we obtain from Taylor's Theorem that
	\begin{align}\label{eqn:fullpower_proof_Taylor_C<infty}
		v_n\Phi_s\Bigl( \frac{ \mathbb{F}_n(v_n)}{ v_n } \Bigr)= \frac{ ( \mathbb{F}_n(v_n)-v_n)^2 }{ 2 v_n } \Bigl( \frac{ \mathbb{F}_n^* }{v_n} \Bigr)^{s-2}, 
	\end{align}
	where $\mathbb{F}_n^*$ is a random variable fulfilling $\min\{v_n,\mathbb{F}_n(v_n)\}\leq \mathbb{F}_n^*\leq \max \{v_n,\mathbb{F}_n(v_n)\}$. We deduce from  \eqref{eqn:fullpower_proof_Chang+Wellner} that 	for all $0<\delta<\min\{1,C^{s-2}\}$
	\begin{align}\label{eqn:fullpower_proof_G}
		Q_n^n\Bigl( \Bigl( \frac{ \mathbb{F}_n^* }{v_n} \Bigr)^{s-2}>\delta \Bigr) \geq Q_n^n\Bigl( \min \Bigl\{ 1, \Bigl( \frac{\mathbb{F}_n(v_n) }{v_n} \Bigr)^{s-2} \Bigr\} > \delta \Bigr) \to 1.
	\end{align} 
	\citet{DitzhausJanssen2017} showed in the proof of their Theorem 4.1 that under \eqref{eqn:fullpower_proof_ass1}
	\begin{align}\label{eqn:fullpower_proof_Chebyschev}
		Q_n^n \Bigl( \sqrt{n} \frac{|\mathbb{F}_n(v_n)-v_n|}{\sqrt{v_n(1-v_n)l_n}} > \gamma \Bigr) \to 1 \text{ for all }\gamma>0.
	\end{align}
	The main idea of proving this is a simple application of Chebyschev's inequality. Combining \eqref{eqn:fullpower_proof_Taylor_C<infty} to \eqref{eqn:fullpower_proof_Chebyschev} yields \eqref{eqn:fullpower_proof_suff}. \\
	Now consider $C=\infty$. From \eqref{eqn:fullpower_proof_nF(v_n)}, \eqref{eqn:fullpower_proof_Chang+Wellner} and \eqref{eqn:fullpower_proof_toC} we obtain for all $\gamma>0$:
	\begin{align}\label{eqn:fullpower_proof_nempF} 
		 Q_n^n(  l_n^{-1}n\mathbb{F}_n(v_n) >\gamma )\to 1
		\text{ and } Q_n^n ( v_n^{-1} \mathbb{F}_n(v_n)  > \gamma ) \to 1.
	\end{align}
	Regarding the second statement of \eqref{eqn:fullpower_proof_nempF} we only need to analyse $\Phi_s(x)$ for sufficiently large $x$  more closely to prove \eqref{eqn:fullpower_proof_suff}. It is easy to verify that there exists some $c_{1,s},c_{2,s}>0$ and $c_{3,s}\in\R$ such that 
	\begin{align}\label{eqn:fullpower_proof_C=infty_Psi(x)=...}
		\Phi_s(x)\geq c_{2,s}x+c_{3,s}\text{ for all }x\geq c_{1,s}.
	\end{align}
	For this purpose consider the cases $s<0,s=0,s\in(0,1),s>1$ separately. E.g., if $s\in(0,1)$ then	$\Phi_s(x) \geq x(s-x^{s-1})/(s(1-s)) \geq x/(2(1-s))$  for all $x\geq (2^{-1}s)^{\frac{1}{s-1}}$. From \eqref{eqn:fullpower_proof_nempF} we get for all $\gamma_0>0$ 
	\begin{align*}
		&Q_n^n\Bigl( \frac{nv_n}{l_n}\Bigl( c_{2,s}\frac{ \mathbb{F}_n(v_n)}{v_n} + c_{3,s} \Bigr)> \gamma_0 \Bigr)\\
		&\geq Q_n^n\Bigl( \frac{n\mathbb{F}_n(v_n)}{l_n}> \frac{2\gamma_0}{c_{2,s}}\,,\, \frac{\mathbb{F}_n(v_n)}{v_n}> \frac{2c_{3,s}}{c_{2,s}} \Bigr) \to 1.
	\end{align*}
	Combining this, the second statement in \eqref{eqn:fullpower_proof_nempF} and \eqref{eqn:fullpower_proof_C=infty_Psi(x)=...} yields \eqref{eqn:fullpower_proof_suff}.
\end{proof}

\begin{proof}[Proof of \Cref{theo:undetect}]
	Similar to the two previous proofs we can assume that $P_0=\lebesgue_{|(0,1)}$. Let $d_n=n^{-1}(\log n)^5$. 
	Keep in mind throughout the whole proof that due to mutual contiguity every statement which holds in $P_0^n$-probability does so in $Q_n^n$-probability and vice versa. Let $\mathbb{Z}_n(a,b)$ be defined as in \eqref{eqn:def_Zn_process}. By Theorem 4.2 of \citet{DitzhausJanssen2017} we have $b_n \mathbb{Z}_n(0,1) -c_n \to Y$ in distribution under $Q_n^n$,  where $b_n$ and $c_n$ are defined by \eqref{eqn:def_bn_cn}. Consequently, all statements in \eqref{eqn:jaeschke_results1} and \eqref{eqn:jaeschke_results2} hold also under $Q_n^n$.  Combining these statements and \eqref{eqn:Sn(2)=....} we obtain that \eqref{eqn:limit_null} holds under the alternative in the case of $s=2$.	Since \eqref{eqn:Rj} and \eqref{eqn:Ks-rn_to_-infty} are also valid in $Q_n^n$-probability we deduce the statement for general $s\in\R$ from \eqref{eqn:Ks()...taylor}.
\end{proof}

\begin{proof}[Proof of \Cref{theo:cai_wu_ext}]
	The proof is divided into two parts. First, we discuss the case $\beta<\beta^*$ by applying our \Cref{theo:full_power}. Second, we discuss the case $\beta>\beta^*$ by applying Lemma 3.12 of \citet{DitzhausJanssen2017}. To improve the readability, set $l_n=\log\log(n)$.
	
	\begin{center}
		\underline{The case $\beta<\beta^*$}	
	\end{center}
	By \Cref{theo:full_power} and \Cref{rem:simp_Hn} it is sufficient to show  that
	\begin{align*}
		l_n^{-1}\widetilde H_n(v_n)&= n^{\frac{1}{2}-\beta}v_n^{-\frac{1}{2}}l_n^{-1}  \Bigl(  \int_0^{v_n} \frac{\mathrm{ d } \mu_{n}}{\,\mathrm{ d }P_{0}}(F_0^{-1}(x)) + \frac{\mathrm{ d } \mu_{n}}{\,\mathrm{ d }P_{0}}(F_0^{-1}(1-x)) \,\mathrm{ d } x \Bigr)
	\end{align*}
	converges to $\infty$ for some $\log(n)n^{-1}\leq v_n\leq(\log(n))^{-1}$. Using the parametrization $v_n=n^{-\tau_n}$ with $\tau_n\in[\widetilde\tau_n,1-\widetilde\tau_n]$ and $\widetilde\tau_n=l_n(\log(n))^{-1}$  we obtain from the substitution $n^{-t}=x$ that 
	\begin{align*}
		l_n^{-1}\widetilde H_n(v_n)\geq n^{\frac{1}{2}-\beta+\frac{1}{2}\tau_n} l_n^{-1}\log(n)\int_{\tau_n}^{\infty} \exp ( h_n(t) -t\log(n) 
		) \,\mathrm{ d }t.
	\end{align*}
	Fix $\delta\in(0,1)$ with $\delta^{-1}\in\N$ and $2\delta\leq\beta^*-\beta$. By \eqref{eqn:cor:cai_wu_ext_Leb_>_0} there exists some $\kappa\in(0,1/2)$ such that for every sufficiently large $n\in\N$ we have
	\begin{align}
		& \lebesgue(\; t\in (1,\infty)  :  ( \beta^*-\delta -1 + t )\log(n)\leq h_{n}(t) \;) \geq \kappa 	\label{eqn:HC_caiwu_def_B1} \\
		\textrm{ or }\;\, &\lebesgue\Bigl(  t\in (0,1) : \Bigl( \beta^*-\delta +\frac{t}{2}-\frac{1}{2} \Bigr)\log(n) \leq h_{n}(t) \Bigr) \geq \kappa. \label{eqn:HC_caiwu_def_B2n}
	\end{align} 
	If \eqref{eqn:HC_caiwu_def_B1} holds then we consider $\tau_n=1-\widetilde\tau_n$ and get 
	\begin{align*}
		l_n^{-1}\widetilde H_n(v_n)\geq \kappa \,\sqrt{\log(n)}l_n^{-1} n^\delta.
	\end{align*}
	Otherwise, if \eqref{eqn:HC_caiwu_def_B2n} holds and $n$ is sufficiently large then 
	\begin{align*}
		\lebesgue\Bigl(  t\in ( \delta( j_n-1)+\widetilde\tau_n,j_n\delta  ) :  \Bigl( \beta^*-\delta +\frac{t}{2}-\frac{1}{2} \Bigr)\log(n) \leq h_{n}(t) \Bigr) \geq \frac{\delta\kappa}{2}
	\end{align*}
	for some $j_n\in\{1,\ldots,\delta^{-1}\}$. In this case set $\tau_n=\delta( j_n-1)+\widetilde\tau_n$ and obtain 
	\begin{align*}
		l_n^{-1}\widetilde H_n(v_n)\geq \frac{\kappa\delta}{2l_n}\log(n) n^{\beta^*-\beta-\delta+\frac{1}{2}\tau_n-\frac{1}{2}j_n\delta}
		\geq\frac{\kappa\delta}{2l_n} (\log(n))^{\frac 32}n^{\frac{1}{2}\delta}.
	\end{align*}
	To sum up, we can conclude that $l_n^{-1}\widetilde H_n(v_n)\to\infty$.

	\begin{center}
		\underline{The case $\beta>\beta^*$}	
	\end{center}
	Fix $x>0$. By Lemma 3.12 of \citet{DitzhausJanssen2017} and the substitution $u=n^{-t}$ it is sufficient to show that 
	\begin{align}
		  &n^{1-\beta}\mu_n\Bigl( n^{-\beta} \frac{\mathrm{ d } \mu_{n}}{\,\mathrm{ d }P_{0}} > x \Bigr)\label{eqn:cai_wu_ext_cond_mu(...} \\
		   &=n^{1-\beta} \int_{\{ u\leq \frac{1}{2}:n^{-\beta} \frac{\mathrm{ d } \mu_{n}}{\,\mathrm{ d }P_{0}}\left( F_{0}^{-1}(u) \right) > x \} } \frac{\mathrm{ d } \mu_{n}}{\,\mathrm{ d }P_{0}}(F_0^{-1}(u)) \,\mathrm{ d }u  \nonumber  \\
		  &+ n^{1-\beta}\int_{\{ u\leq \frac{1}{2}:n^{-\beta} \frac{\mathrm{ d } \mu_{n}}{\,\mathrm{ d }P_{0}}\left( F_{0}^{-1}(1-u) \right) > x \} } \frac{\mathrm{ d } \mu_{n}}{\,\mathrm{ d }P_{0}}(F_0^{-1}(1-u)) \,\mathrm{ d }u\nonumber \\
		  &\leq 2 n^{1-\beta}\log(n) \int_{\{ t\geq \frac{\log(2)}{\log(n)}:n^{-\beta} \exp(h_n(t)) > x \} } \exp(h_n(t)-t\log(n)) \,\mathrm{ d }t \nonumber \\
		  \intertext{ converges to $0$ and }\;
		  &n^{1-2\beta} \int_{\{n^{-\beta} \frac{\mathrm{ d } \mu_{n}}{\,\mathrm{ d }P_{0}}\leq x\}} \Bigl( \frac{\mathrm{ d } \mu_{n}}{\,\mathrm{ d }P_{0}} \Bigr)^2\,\mathrm{ d }P_0\nonumber \\
		  &\leq 2 n^{1-2\beta}\log(n) \int_{\{ t\geq \frac{\log(2)}{\log(n)}:n^{-\beta} \exp(h_n(t)) \leq x \} } \exp(2h_n(t)-t\log(n)) \,\mathrm{ d }t \nonumber 
	\end{align}
	does so as well. To verify this, set
	\begin{align*}
		& I_{n,1}= \log(n)n^{1-2\beta}  \int_{ \frac{\log(2)}{\log(n)}}^1
		\exp\Bigl( \log(n) \Bigl( \frac{2 h_{n}(t)}{\log(n)} -  t \Bigr)  \Bigr)\,\mathrm{ d } t, \\ 
		&I_{n,2}= \log(n)n^{1-\beta} \int_{1}^{M}
		\exp\Bigl( \log(n) \Bigl( \frac{h_{n}(t)}{\log(n)} -  t \Bigr)  \Bigr)\,\mathrm{ d } t \\
		\textrm{and }& I_{n,3}= \log(n)n^{1-\beta} \int_{M}^\infty
		\exp\Bigl( \log(n) \Bigl( \frac{h_{n}(t)}{\log(n)} -  t \Bigr)  \Bigr)\,\mathrm{ d } t. 
	\end{align*}
	Observe that
	\begin{align*}
		&n^{1-\beta}\mu_n\Bigl( n^{-\beta} \frac{\mathrm{ d } \mu_{n}}{\,\mathrm{ d }P_{0}} > x \Bigr) \leq 2x^{-1}I_{n,1}+2(I_{n,2}+I_{n,3}) \\
		\text{and }&
		n^{1-2\beta} \int_{\{n^{-\beta} \frac{\mathrm{ d } \mu_{n}}{\,\mathrm{ d }P_{0}}\leq x\}} \Bigl( \frac{\mathrm{ d } \mu_{n}}{\,\mathrm{ d }P_{0}} \Bigr)^2\,\mathrm{ d }P_0 \leq 2I_{n,1}+2x(I_{n,2}+I_{n,3}).
	\end{align*}
	From \eqref{eqn:cor:cai_wu_ext_Leb_to_0} with $\delta=(\beta-\beta^*)/2>0$ we deduce for all $n\geq N_{1,\delta}$ that $I_{n,1}\leq \log(n)n^{-2\delta}\to 0$ and $I_{n,2}\leq \log(n)n^{-\delta}M\to 0$.  Consequently, it remains to show that $I_{n,3}\to 0$.\\
	First, assume that \eqref{eqn:cor:cai_wu_ext_techn_cond_lebes} holds for $\delta=(\beta-\beta^*)/2>0$. Then for all $n\geq N_{2,\delta}$
	\begin{align*}
		I_{n,3} \leq n^{-\delta} \log(n) \int_{1}^\infty 
		\exp( -\lambda_n t  )\,\mathrm{ d } t =\log(n)n^{-\delta}\lambda_n^{-1}\to 0.
	\end{align*}
	Second, suppose that \eqref{eqn:cor:cai_wu_ext_techn_cond_uniform} is fulfilled. Similar to the calculation in \eqref{eqn:cai_wu_ext_cond_mu(...}, we obtain $\log(n)\int \exp(h_{n,j}(t))n^{-t}\,\mathrm{ d }t=1$ and, thus, $\int \exp(h_n(t))n^{-t}\,\mathrm{ d } t\leq 2$. For all $\kappa>0$ there exists some $ N_{3,\kappa}\in\N$ such that for all $n\geq N_{3,\kappa}$ we have
	\begin{align}\label{eqn:caiwu_sup}
		\sup_{t\geq M}\Bigl|\gamma(t)-\frac{h_n(t)}{\log(n)}\Bigr|\leq \kappa \text{ and, hence, }
		\int_M^\infty n^{\gamma(t)-t} \,\mathrm{ d }t \leq  2n^\kappa .
	\end{align}
	From the latter and a simple proof by contradiction we can deduce that
	\begin{align*}
		\lebesgue(t\geq M:\gamma(t)-t>0)=0.
	\end{align*}
	We want to point out that \citet{Cai_Wu_2014} already showed the two previous statements for $M=1$ under the assumption \eqref{eqn:theo:Cai+Wu_con}. Let $\lfloor x\rfloor=\max\{k\in\N:k\leq x\}$ be the integer part of $x$ and $\tau_n=l_n/\log(n)$. To show $I_{n,3}\to 0$ we divide $h_n(t)$ as follows: $h_n(t)=(1-\tau_n)h_n(t)+\tau_nh_n(t)$. To get an upper bound, we use  \eqref{eqn:cor:cai_wu_ext_Leb_to_0} with $\delta=(\beta-\beta^*)/2$ for the first summand and the first statement in \eqref{eqn:caiwu_sup} with $\kappa=1$ for the second summand. Consequently, there is some $N_4\geq N_{1,\delta}+\exp(N_{3,1})$ such that for all $n\geq N_4$
	\begin{align*}
		I_{n,3} &\leq \log(n) n^{1-\beta+(1-\tau_n)(\beta^*+\delta-1)} \int_M^\infty n^{\tau_n(\gamma(t)+1-t)}\,\mathrm{ d }t\\
		&\leq \log(n)^{3-\beta^*-\delta} n^{-\delta} \int_M^\infty \lfloor \log(n)\rfloor^{\gamma(t)-t}\,\mathrm{ d }t
		\leq n^{-\frac{\delta}{2}} 2\lfloor \log(n)\rfloor\to 0. 
	\end{align*}
\end{proof}

\begin{proof}[Proof of \Cref{theo:cai_wu_expfam}]
	Without loss of generality we can assume that $P_0=\lebesgue_{|(0,1)}$ and $T(0)=0$. By assumption $T$ restricted on $[0,\eta]$ is invertible. Denote by $T^{-1}$ its inverse. We deduce from Theorem 1.5.12 of \citet{BinghamETAL1987} that for all $x\in[0,-T(\eta)]$ we have
	\begin{align*}
		\lebesgue_{|(0,1)}^{-T}(0,x)=T^{-1}(x)=x^pL_1\Bigl( \frac{1}{x} \Bigr)
	\end{align*}
	for a slowly varying function $L_1$ at infinity. Hence, from Theorems XIII.5.2 and XIII.5.3 of \citet{Feller1966} we obtain $C(n^{r})= n^{rp}L_2(n^r)$ for a slowly varying function $L_2$ at infinity. Moreover, it is well known, see Proposition 1.3.6 of \citet{BinghamETAL1987}, that $\log(L_2(x))=o(\log(x))$ and $L(x)x^{\kappa}\to\infty$ as $x\to\infty$. Let  $h_{n,1}$, $h_{n,2}$ and $h_n$ be defined as in \Cref{theo:Cai+Wu}. Fix $\delta\in(0,rp)$ and set $\lambda_n=\log\log(n)$. By the monotonicity of $T$
	we have 
	\begin{align*}
		h_n(t) \leq \log(C(\vartheta_n)) + \vartheta_n T(n^{-rp+\frac{\delta}{2}})= (rp+o(1)) \log(n)  - n^{\frac{\delta}{2p}}L(n^{rp})
	\end{align*}
	for all $t\leq rp-\delta/2$. Consequently, there is some constant $K>0$ such that
	\begin{align*}
		&\sup_{t\in(\log(2)/\log(n), rp-\frac{\delta}{2}) } \Bigl\{\frac{h_n(t)}{\log(n)} - \Bigl( 1-\frac{\lambda_n}{\log(n)}\Bigr)t+\frac{\min\{1,t\}}{2} \Bigr\}
		\leq K -n^{\frac{\delta}{4p}} \to -\infty.
	\end{align*}
	Furthermore, since $T\leq 0$ we have for sufficiently large $n$
	\begin{align*}
		&\sup_{t\geq rp-\frac{\delta}{2} } \Bigl\{\frac{h_n(t)}{\log(n)} - \Bigl( 1-\frac{\lambda_n}{\log(n)}\Bigr)t+\frac{\min\{1,t\}}{2} \Bigr\} \\
		&\leq rp + o(1) - \Bigl( rp-\frac{\delta}{2} \Bigr)  + \frac{\min\{rp-\delta/2,1\}}{2}<\beta^\#(r,p) - \frac{1}{2} + \delta.
	\end{align*}
	To sum up, \eqref{eqn:cor:cai_wu_ext_Leb_to_0} and \eqref{eqn:cor:cai_wu_ext_techn_cond_lebes}  hold for $\beta^*=\beta^\#(r,p)$. Similar to the previous calculations, 
	\begin{align*}
		\sup_{t\in(rp+\delta/4,rp+\delta/2) } \Bigl\{\frac{h_n(t)}{\log(n)} - t+\frac{\min\{1,t\}}{2} \Bigr\}> \beta^\#(r,p)-\frac{1}{2}-\delta
	\end{align*}
	for sufficiently large $n$	and so \eqref{eqn:cor:cai_wu_ext_Leb_>_0} holds for $\beta^*=\beta^\#(r,p)$.
\end{proof}

\begin{proof}[Proof of \Cref{theo:cai_wu_ext_norm}]
	Due to the analogy to the proof of \Cref{theo:cai_wu_ext} we skip some details here. Set $l_n=\log\log(n)$. By $\Phi$, $\Phi^{-1}$, $\phi$ we denote the distribution function, the left continuous quantile function and the density of $N(0,1)$, respectively.
	
	\begin{center}
		\underline{The case $\beta<\beta^*$}	
	\end{center}
	By $\Phi^{-1}(1-x)=-\Phi^{-1}(x)$, \Cref{theo:full_power} and \Cref{rem:simp_Hn} it remains to show for some $v_n=n^{-\tau_n}$ with $\tau_n\in[\widetilde\tau_n,1-\widetilde\tau_n]$ and $\widetilde\tau_n=l_n/\log(n)$ that
	\begin{align*}
	l_n^{-1}\widetilde H_n(v_n)= n^{\frac{1}{2}-\beta+\frac{\tau_n}{2}} \frac{1}{l_n\sqrt{2\pi}} \int \frac{\mathrm{ d } \mu_{n}}{\,\mathrm{ d }P_{0}}(x) \exp\Bigl( -\frac{x^2}{2} \Bigr)\mathbf{1}\Bigl\{ |x|\geq \Phi^{-1}(1-v_n) \Bigr\}\,\mathrm{ d }x 
	\end{align*}
	converges to $\infty$. A simple consequence of integration by parts is $\phi(x)(x/(1+x^2))\leq 1- \Phi(x) \leq \phi(x)/x$ for all $x>0$. By this it is easy to obtain $\Phi^{-1}(1-u)\leq \sqrt{-2\log(u)}$ for all sufficiently small $u>0$ and so $\Phi^{-1}(1-v_n)\leq \sqrt{2\tau_n\log(n)}$ for all sufficiently large $n\in\N$. Combining this and the substitution $x=y\sqrt{2\log(n)}$  yields 
	\begin{align*}
	l_n^{-1}\widetilde H_n(v_n)\geq   \int n^{\frac{1}{2}-\beta+\frac{1}{2}\tau_n+\frac{\widetilde h_n(y)}{\log(n)}-y^2} \mathbf{1}\Bigl\{|y|\geq \sqrt{\tau_n}\Bigr\}\,\mathrm{ d }y
	\end{align*}
	for sufficiently large $n\in\N$. 
	Fix $\delta\in(0,1)$ with $\delta^{-1}\in\N$ and $2\delta\leq \beta^*-\beta$. By \eqref{eqn:cor:cai_wu_norm_Leb_>_0} there exists some $\kappa\in(0,1/2)$ such that for every sufficiently large $n\in\N$ we have
	\begin{align}
		& \lebesgue\Bigl(\; |y|>1  :   \beta^*-\delta -1 + y^2 \leq \frac{\widetilde h_{n}(y)}{\log(n)} \;\Bigr) \geq \kappa 	\label{eqn:HC_caiwu_norm_def_B1} \\
	\textrm{ or } &	\lebesgue\Bigl(  |y|\in ( \sqrt{\delta( j_n-1)+\widetilde\tau_n},\sqrt{j_n\delta}  ) :  \beta^*-\delta +\frac{y^2-1}{2}  \leq \frac{\widetilde h_{n}(y)}{\log(n)} \Bigr) \geq \kappa \nonumber 
	\end{align} 
	for some appropriate $j_n\in\{1,\ldots,\delta^{-1}\}$. If \eqref{eqn:HC_caiwu_norm_def_B1} holds then set $\tau_n=1-\widetilde\tau_n$ and otherwise set $\tau_n=\delta(j_n-1)+\widetilde{\tau}_n$. Consequently, we obtain analogously to the proof of \Cref{theo:cai_wu_ext} that  $l_n^{-1}\widetilde H_n(v_n)\to \infty$.
	
	\begin{center}
		\underline{The case $\beta>\beta^*$}	
	\end{center}
	Set
	\begin{align*}
	& I_{n,1}= n^{1-2\beta}\sqrt{\frac{\log(n)}{\pi}}  \int_{-1}^1
	n^{\frac{2 \widetilde h_{n}(x)}{\log(n)} -  x^2}\,\mathrm{ d } x, \\ 
	&I_{n,2}= n^{1-\beta}\sqrt{\frac{\log(n)}{\pi}}  \int
	n^{\frac{ \widetilde h_{n}(x)}{\log(n)} -  x^2}\mathbf{1}\{|x|\in(1,M)\}\,\mathrm{ d } x\\
	\textrm{and }& I_{n,3}= n^{1-\beta}\sqrt{\frac{\log(n)}{\pi}}  \int
	n^{\frac{ \widetilde h_{n}(x)}{\log(n)} -  x^2}\mathbf{1}\{|x|\geq M\}\,\mathrm{ d } x. 
	\end{align*}
	Fix $y>0$. Set $P_0=N(0,1)$.  As done in the proof of \Cref{theo:cai_wu_ext}, it remains to verify that 
	\begin{align*}
	&n^{1-\beta}\mu_n\Bigl( n^{-\beta} \frac{\mathrm{ d } \mu_{n}}{\,\mathrm{ d }P_0} > y \Bigr) \leq \frac{1}{y} I_{n,1} + I_{n,2}+I_{n,3}\\
	\text{and }\; &n^{1-2\beta} \int \mathbf{1}\Bigl\{n^{-\beta} \frac{\mathrm{ d } \mu_{n}}{\,\mathrm{ d }P_0}\leq y\Bigr\} \Bigl( \frac{\mathrm{ d } \mu_{n}}{\,\mathrm{ d }P_0} \Bigr)^2\,\mathrm{ d }P_0 \leq  I_{n,1}+y(I_{n,2}+I_{n,3})
	\end{align*}
	converges to $0$, i.e. $I_{n,1},I_{n,2}$ and $I_{n,3}$ converge to $0$. From \eqref{eqn:cor:cai_wu_norm_Leb_to_0} with $\delta=(\beta-\beta^*)/2>0$ we deduce that $I_{n,1}\leq 2n^{-2\delta}\sqrt{\log(n)/\pi}\to 0$ and $I_{n,2}\leq 2M n^{-\delta}\sqrt{\log(n)/\pi}\to 0$. It remains to discuss $I_{n,3}$. First, assume that \eqref{eqn:cor:cai_wu_norm_techn_cond_lebes} holds for $\delta=(\beta-\beta^*)/2>0$. Then
	\begin{align*}
	I_{n,3} \leq n^{-\delta} \sqrt{\frac{\log(n)}{\pi}} \int \exp(-\lambda_nx^2)\,\mathrm{ d }x = n^{-\delta} \sqrt{\frac{\log(n)}{\lambda_n}} \to 0.
	\end{align*}
	Second, suppose that \eqref{eqn:cor:cai_wu_norm_techn_cond_uniform} is fulfilled. Analogously to the proof of \Cref{theo:cai_wu_ext}, there is some $\widetilde N_\kappa\in\N$ such that $\int \mathbf{1}\{|x|\geq M\} n^{\alpha(x)-x^2} \,\mathrm{ d }x\leq 2 n^{\kappa}$ for every $\kappa>0$ and all $n\geq \widetilde N_\kappa$. Moreover, $\lebesgue(|x|\geq M:\alpha(x)-x^2>0)=0$. Let $\tau_n=l_n/\log(n)$. Finally, from \eqref{eqn:cor:cai_wu_norm_Leb_to_0} with $\delta=(\beta-\beta^*)/2$ and \eqref{eqn:cor:cai_wu_norm_techn_cond_uniform} we get for sufficiently large $n$
	\begin{align*}
	I_{n,3} &\leq \sqrt{\frac{\log(n)}{\pi}} n^{1-\beta+(1-\tau_n)(\beta^*+\delta-1)} \int \mathbf{1}\{|x|\geq M\} n^{\tau_n(\alpha(x)+1-x^2)}\,\mathrm{ d }x\\
	&\leq n^{-\frac{\delta}{2}} \int \mathbf{1}\{|x|\geq M\} \lfloor \log(n)\rfloor^{\alpha(x)-x^2}\,\mathrm{ d }x
	\leq n^{-\frac{\delta}{2}} 2\lfloor \log(n)\rfloor\to 0. 
	\end{align*}
\end{proof}

\begin{proof}[Proof of \Cref{theo:den_exp}]
	We split the proof into two steps:
	\begin{center}
		\underline{1. Likelihood ratio test sequences:}
	\end{center}
	Let $r\geq \rho^*(\beta)$. Here, we give proof for \eqref{enu:theo:den_exp_r=} regarding to the likelihood ratio tests as well as  \eqref{enu:theo:den_exp_r>}. First, we introduce the variational distance $||P-Q||$ between to probability measures $P$ and $Q$ on the same measure space:
	\begin{align*}
		||P-Q||= \frac{1}{2}\int \Bigl | \frac{\mathrm{ d }P}{\mathrm{ d }(P+Q)}-\frac{\mathrm{ d }Q}{\mathrm{ d }(P+Q)} \Bigr | \,\mathrm{ d }(P+Q).
	\end{align*}
	By Lemmas 2.2 and 2.3 of \citet{Strasser1985} we have for every fixed $n\in\N$ that the sharp lower bound of the sum of error probabilities for all tests testing $P_0^n$ versus $Q_n^n$ is equal to $1-||P_0^n-Q_n^n||$. Moreover, this bound is attained by the likelihood ratio test $\varphi_n=\mathbf{1}\{ \mathrm d Q_n^n/\mathrm dP_0^n\geq 1\}$. It is well known and easy to show that weak convergence of binary experiments $\{P_0^n,Q_n^n\}\xrightarrow{w}\{P,Q\}$ implies convergence of the variational distance $||P_0^n-Q_n^n||\to ||P-Q||$. For more details about the convergence of binary or more general experiments we refer the reader to the book of \citet{Strasser1985}. To sum up, it is sufficient to show that $\{P_0^n,Q_n^n\}$ converges weakly to the uninformative experiment $\{\epsilon_0,\epsilon_0\}$ (if $r > \rho^{*}(\beta)$) and to the normal shift experiment  $\{N(-\sigma^2/2,\sigma^2),N(\sigma^2/2,\sigma^2)\}$ with $\sigma^2=\text{Var}_{P_{(0)}}(T)$ (if $r= \rho^{*}(\beta)$), respectively. For all $x>0$ define
	\begin{align*}
		&I_{n,1,x}=n^{1-\beta}P_{(\vartheta_n)}\Bigl( \varepsilon_{n}\frac{\mathrm{ d } P_{(\vartheta_n)}}{\,\mathrm{ d }P_{(0)}} > x  \Bigr),\\
		&I_{n,2,x}=n^{1-2\beta} \Bigl( \frac{C(\vartheta_n)^2}{C(2\vartheta_n)}P_{(2\vartheta_n)} \Bigl( \varepsilon_{n} \frac{\mathrm{ d } P_{(\vartheta_n)}}{\,\mathrm{ d }P_{(0)}}\leq x  \Bigr)-1 \Bigr).
	\end{align*}
	By Theorem 3.10 and Lemma 3.12 of \citet{DitzhausJanssen2017} it remains to show that 
	\begin{align}\label{eqn:den_exp_suff_I}
		&I_{n,1,x}\to 0 \;\text{ and }\; I_{n,2,x}\to \sigma^2 \mathbf{1}\{r=\rho^*(\beta)\}\text{ with }\sigma^2=\text{Var}_{P_{(0)}}(T)
	\end{align}
	for all $x>0$. For this purpose we introduce the Laplace transform $\omega$ defined
	\begin{align}\label{eqn:Laplace_transf}
		\omega({\vartheta})=C({\vartheta})^{-1}=\int \exp \bigl( {\vartheta}{T}\bigr)\,\mathrm{ d }P_{(0)},\;\vartheta\in\Theta.
	\end{align}
	By Corollary 7.1 of \citet{Barndorff-Nielsen} the Laplace transform $\omega$ is analytic in a neighborhood around ${0}$ and the derivatives can be determined by differentiation under the integral sign. Hence, there is $M\in(1,\infty)$ such that for all $x>0$ and $n\geq N_x$ we have  
	\begin{align*}
		& C(2{\vartheta}_n)\leq 2, \,\,\,\frac{\varepsilon_{n}C({\vartheta}_n)}{x} \leq  e^{-1} 
		\text{ and }
		\omega^{(4)} (2\vartheta_n)= \int T^4 \exp( 2\vartheta_n T) \,\mathrm{ d }P_{(0)}<M,
	\end{align*}
	where $f^{(k)}$ denotes the derivative of order $k\in\N$ of the function $f$. 	From this we obtain for all $x>0$ and $n\geq N_x$ that
	\begin{align*}
		 P_{(2\vartheta_{n})}\left(  \varepsilon_{n} \frac{\mathrm{ d } P_{(\vartheta_n)}}{\,\mathrm{ d }P_{(0)}}> x \right) &\leq 2 \int \mathbf{1}\Bigl\{ \vartheta_n T > \log \Bigl( \frac{x}{\varepsilon_{n}C(\vartheta_n)}\Bigr) \Big\} \exp ( 2\vartheta_n T ) \mathrm d P_{(0)}\\
		 &\leq  2\int \, (\vartheta_nT)^{4}\, \exp\left( 2\vartheta_n T \right)\,\mathrm{ d } P_{(0)} = o\bigl( n^{-2r}\bigr).
	\end{align*}
	By Taylor expansion around $0$ we get as $\vartheta\to 0$
	\begin{alignat*}{2}
		&\omega(2\vartheta)&&= 1 + 2 \vartheta E_{P_{(0)}}(T)  + 2 \vartheta^2 E_{P_{(0)}}(T^2)  + o\bigl( \vartheta^2 \bigr)  \textrm{ and so}\\
		&\omega(\vartheta)^2 
		&&= \omega(2\vartheta)-\vartheta^2\sigma^2+ o\bigl( \vartheta^2 \bigr).
	\end{alignat*}
	Consequently, for all $x>0$ 
	\begin{align*}
		I_{n,2,x} =n^{1-2\beta}\Bigl( \frac{\omega(2\vartheta_n)}{\omega(\vartheta_n)^2} (1+o(\vartheta_n^2))-1 \Bigr)	= n^{1-2\beta-2r} \sigma^2(1+ o( 1 ) ),
	\end{align*}
	which proves the statement about $I_{n,2,x}$ in \eqref{eqn:den_exp_suff_I}. Furthermore, the one about $I_{n,1,x}$ is fulfilled since for all $x>0$
	\begin{align*}
		I_{n,1,x}&\leq \frac1x n^{1-2\beta} \frac{C( \vartheta_{n} )^2}{C( 2\vartheta_{n} )} P_{(2\vartheta_n)}\Bigl( \varepsilon_{n} \frac{\mathrm{ d } P_{(\vartheta_n)}}{\,\mathrm{ d }P_{(0)}}> x \Bigr)	= o\bigl(n^{1-2\beta-2r} \bigr).
	\end{align*}
	\begin{center}
		\underline{2. Test sequences based on $S_n(s)$:}	
	\end{center}
	Here, we give proof for \eqref{enu:theo:den_exp_r=} corresponding to the results about $S_n(s)$ and \eqref{enu:theo:den_exp_r<}. Therefore, we apply \Cref{theo:full_power,theo:undetect}. This requires a closer analysis of $H_n(v)$ given by \eqref{eqn:def_H}. Without loss of generality we can assume $P_{(0)}=\lebesgue_{|(0,1)}$. Observe that for all $v\in(0,1/2)$ 
	\begin{align*}
		&\mu_n(0,v)-v = \int_{B_{1,v}} \chi_x(\vartheta_n) \,\mathrm{ d }x \text{ and }\mu_n(1-v,1)-v = \int_{B_{2,v}} \chi_x(\vartheta_n) \,\mathrm{ d }x,\\
		&\text{where }\chi_x(\vartheta)=C(\vartheta)\exp(\vartheta T)-1, \;B_{1,v}=(0,v)\text{ and }B_{2,v}=(1-v,1).
	\end{align*}
	As already stated, the Laplace transform $\omega$ introduced in \eqref{eqn:Laplace_transf} is analytic in $(-\delta,\delta)$ for sufficiently small $\delta\in(0,1)$ and so is $C$. For all $\vartheta\in(-\delta,\delta)$ and every $x\in(0,1)$ we have
	\begin{align*}
		\chi_x^{(1)}(\vartheta)&=\; C^{(1)}(\vartheta)\exp(\vartheta T(x)) + C(\vartheta) T(x) \exp(\vartheta T(x))  \\
		\textrm{and }\;\chi_x^{(2)}(\vartheta)&=\; \left[ C^{(2)}(\vartheta) + 2C^{(1)}(\vartheta) T(x) + C(\vartheta)T^2(x) \right]\exp(\vartheta T(x)).
	\end{align*}
	Since $C^{(1)}(0)=-\omega^{(1)}(0)=-\int_0^1T\,\mathrm{ d }\lebesgue$ we get from a Taylor expansion around $0$ that for all sufficiently large $n$, such that $\vartheta_n<\delta$, we have
	\begin{align*}
		\chi_x(\vartheta_n) = \vartheta_n \Bigl(-\int_0^1 T \,\mathrm{ d }\lebesgue + T(x) \Bigr) + \frac{\vartheta_n^2}{2}\chi_x^{(2)}(r_n(x)) \text{ with }r_n(x)\in[0,\vartheta_n].
	\end{align*}
	Since $C$, $\omega$ are analytic there exists $M>1$ such that for all $\vartheta\in[-\delta/2,\delta/2]$
	\begin{align*}
		|C(\vartheta)|+|C^{(1)}(\vartheta)|+|C^{(2)}(\vartheta)| +\sum_{k=0}^{2}\int_0^1 |T|^{k} \exp(\vartheta T)\,\mathrm{ d }\lebesgue  \leq M.
	\end{align*}
	By H{\"o}lder's inequality it holds for all $f:(0,1)\to\R$ with $\int_0^1 f^4\,\mathrm{ d }\lebesgue\leq M$ that
	\begin{align*}
		\int_{B_{j,v}} |f| \,\mathrm{ d }\lebesgue \leq v^{\frac{3}{4}} M^{\frac{1}{4}}\leq v^{\frac{3}{4}} M\text{ for }j=1,2.
	\end{align*}
	Hence, for all $j\in\{1,2\}$, $v\in(0,1/2)$ and large $n$, such that $|\vartheta_n|\leq \delta/2$, we get
	\begin{align*}
		\int_{B_{j,v}} | \chi_x^{(2)}(r_n(x))| \,\mathrm{ d }x 
		&\leq M  \int_{B_{j,v}}  \Bigl( 1+2|T|+T^2 \Bigr) \Bigl( \exp(\vartheta_nT)+\exp(-\vartheta_nT) \Bigr) \,\mathrm{ d } \lebesgue\Bigr) \\
		&\leq 8v^{\frac{3}{4}}M^2. 
	\end{align*}
	Consequently,
	\begin{align}\label{eqn:exp_dens_Hn_geq}
		H_n(v) &\geq  n^{\frac{1}{2}-\beta} v^{-\frac{1}{2}}\Bigl | \int_{B_{j,v}} \chi_x(\vartheta_n) \,\mathrm{ d }x \Bigr|\nonumber \\
		&= n^{\frac{1}{2}-\beta-r} v^{-\frac{1}{2}}\Bigl | \int_{B_{j,v}} \Bigl( T-\int_0^1 T\,\mathrm{ d }\lebesgue \Bigr)\,\mathrm{ d }\lebesgue + o(1)  \Bigr|.
	\end{align}
	If $\int_{0}^v ( T-\int_0^1 T\,\mathrm{ d }\lebesgue )\,\mathrm{ d }\lebesgue=0=\int_{1-v}^1 ( T-\int_0^1 T\,\mathrm{ d }\lebesgue )\,\mathrm{ d }\lebesgue$ would hold for every $v\in(0,1/2)$ then $T\equiv \int_0^1 T\,\mathrm{ d }\lebesgue$ would be true $\lebesgue_{|(0,1)}$-almost surely, which contradicts the assumption $\text{Var}_{P_{(0)}}(T)>0$. Consequently, we can deduce from \eqref{eqn:exp_dens_Hn_geq} that $(\log\log(n))^{-1} H_n(v^*)\to \infty$ for some $v^*\in(0,1/2)$ if $r<\rho^*(\beta)$. Applying \Cref{theo:full_power} with $v_n=v^*$ gives us  \eqref{enu:theo:den_exp_r<}. \\
	It remains to discuss the case $r=\rho^*(\beta)$. Set $c_{n,1}=1/n$, $c_{n,2}=c_{n,3}=1/\sqrt{n}$ and $c_{n,4}=(\log(n))^{-4}$. Clearly, \eqref{eqn:undetect_constants}  holds with $\kappa=1/2$. Moreover, we can deduce from our previous considerations that for all sufficiently large $n\in\N$, such that $\vartheta_n\leq \delta/2$, we have
	\begin{align*}
		\sup_{v \in [c_{n,1},c_{n,4}]}H_n(v)
		&\leq   2n^{\frac{1}{2}-\beta-r}  \sup_{v \in [c_{n,1},c_{n,4}]} v^{-\frac{1}{2}} \Bigl ( vM + v^{\frac{3}{4}}M +  8\frac{\vartheta_n}{2}v^{\frac{3}{4}}M^2\Bigr)\\
		& \leq 2\sup_{v \in [c_{n,1},c_{n,4}]} v^{\frac{1}{4}} 10 M^2 = 20M^2  (\log(n))^{-1}.
	\end{align*}
	Hence, \eqref{eqn:undetect_Hn} is fulfilled. Recall that by the first step $\{P_0^n,Q_n^n\}$ converges weakly to $\{N(-\sigma^2/2,\sigma^2),N(\sigma^2/2,\sigma^2)\}$ for some $\sigma^2>0$ and, hence, by the first Lemma of Le Cam $P_0^n$ and $Q_n^n$ are mutually contiguous. Finally, applying \Cref{theo:undetect} yields the corresponding statement in \eqref{enu:theo:den_exp_r=}.
\end{proof}


\end{document}